\documentclass[11pt,twoside]{article}
\usepackage[utf8]{inputenc}
\usepackage{url,amsmath,enumerate,fancyhdr,amssymb,amsthm,lscape,ifpdf,graphicx,float,mathtools}
\IfFileExists{pstricks.sty}{\usepackage{pstricks}}{}
\IfFileExists{epsf.sty}{\usepackage{epsf}}{}
\IfFileExists{bbm.sty}{\usepackage{bbm}}{}
\IfFileExists{soul.sty}{\usepackage{soul}}{}
\DeclarePairedDelimiter{\ceil}{\lceil}{\rceil}
\DeclarePairedDelimiter{\floor}{\lfloor}{\rfloor}

\DeclareMathOperator*{\argmin}{argmin}
\DeclareMathOperator{\dist}{dist}

\DeclareMathOperator{\proj}{proj}
\usepackage{blindtext,relsize}
\usepackage[dvipsnames]{xcolor}

\colorlet{myblue}{blue!70!black}

\usepackage{algorithm,algpseudocode}
\usepackage[section]{placeins}
\usepackage{geometry}
\usepackage{setspace}
\numberwithin{equation}{section}
\usepackage[title]{appendix}
\newcommand{\norm}[1]{\left\lVert#1\right\rVert}

\renewcommand{\arraystretch}{1.2}

\mathchardef\mhyphen="2D
\newtheorem{theorem}{Theorem}[section]
\newtheorem{lemma}[theorem]{Lemma}
\theoremstyle{definition}
\newtheorem{definition}[theorem]{Definition}
\newtheorem{example}[theorem]{Example}

\newtheorem{proposition}[theorem]{Proposition}

\newtheorem{remark}[theorem]{Remark}
\usepackage{booktabs,tabularx,threeparttable}

\newcolumntype{Y}{>{\raggedright\arraybackslash}X}
\newcolumntype{Z}{>{\centering\arraybackslash}X}
\newcolumntype{P}[1]{>{\centering\arraybackslash}p{#1}}

\makeatletter
\DeclareMathOperator*{\infd}{inf\vphantom{\operator@font p}}
\DeclareMathOperator*{\supx}{su\smash{\operator@font p}}
\makeatother
\usepackage[numbers]{natbib}
\allowdisplaybreaks[4]
\usepackage{hyperref}
\hypersetup{colorlinks,citecolor=blue,filecolor=blue,linkcolor=black,urlcolor=blue,
 pdftitle={branch-and-filter},
 pdfsubject={Global convergence and iteration complexity}}
\newcommand{\R}{\overline{\mathbb {R}}}
\newcommand{\Z}{\mathbb Z}
\newcommand{\F}{\mathcal F}
\newcommand{\Ocal}{\mathcal O}
\newcommand{\Acal}{\mathcal A}
\newcommand{\Gcal}{\mathcal G}
\newcommand{\Lcal}{\mathcal L}
\newcommand{\Ical}{\mathcal I}
\newcommand{\Mcal}{\mathcal M}

\newcommand{\eps}{\varepsilon}

\newcommand{\res}{\mathrm{res}}

\newcommand{\pos}[1]{\left[#1\right]_+}

\algrenewcommand\algorithmicrequire{\textbf{Input:}}
\algrenewcommand\algorithmicensure{\textbf{Output:}}
\begin{document}
\title{A branch-and-filter method for mixed-integer linear programming}
\author{Nikos Dimou\footnote{Email: dimou@unc.edu}}
\date{University of North Carolina at Chapel Hill\\Department of Statistics and Operations Research}
\maketitle
\pagestyle{plain}

\vspace{-0.4cm}
\begin{abstract}

We develop a new tree-search framework for mixed-integer linear programming problems, which incorporates a filter mechanism as a node-acceptance device. The introduced \textit{branch-and-filter} algorithm combines proper branching actions with a filter on integrality violation and an incumbent-dependent objective coordinate. Inspired by classical filter methods in nonconvex constrained optimization, the proposed method moves to a restoration phase whenever sufficient progress  cannot be provided by the filter and funnel mechanisms, the latter of which is anchored at global relaxation lower bounds. Under mild assumptions, the introduced method either certifies infeasibility in finitely many iterations, or converges to global optimality. Moreover, we establish a total outer iteration complexity bound of order $\mathcal{O}(\varepsilon^{-1})$ to certification of an $\varepsilon$-optimal mixed-integer solution through fixed-integer linear programming recovery. We finally perform several numerical experiments to study the performance of branch-and-filter against benchmark solvers, under a branch-and-bound, a branch-and-cut, and a novel cell-deletion realization of the proposed restoration oracle.
\end{abstract}

\section{Introduction}\label{sec:intro}
We consider the mixed-integer linear programming (MILP) problem
\begin{equation}\label{eq:milp}
 \begin{aligned}
    &\inf_{x}\;\; c^\top x \\
    &\;\text{s.t.} \;Ax\leq b \\
    &\qquad x\in \mathbb{Z}^{p}\times \mathbb{R}^{q}
 \end{aligned}
\end{equation}
with $p\geq 1$, $q\geq 0$ and rational parameters $A\in \mathbb{Q}^{m\times (p+q)}$, $b\in\mathbb{Q}^{m}$ and $c\in \mathbb{Q}^{p+q}$. Over the past several decades, numerous methods have been developed for solving \eqref{eq:milp}, which offer global optimality guarantees.  The classical branch-and-bound (BB) framework \citep{land2009automatic} combines valid relaxation bounds with branching decisions that preserve feasible solutions. Branch-and-cut (BC) strengthens these bounds using valid inequalities \citep{balas1996mixed}, whereas branch-and-price computes them through column generation \citep{barnhart1998branch}, the latter of which enjoy finite termination properties \citep{vanderbeck1996exact}. Related cutting-plane methods admit explicit finite-convergence proofs for  MILPs with bounded integer variables \citep{balas1993lift,chen2011finite}. Global optimality warranties also arise from various decomposition-based methods, which partition \eqref{eq:milp} into lower-dimensional MILP subproblems \cite{benders1962partitioning, sun2024decomposition}. 

The effectiveness and convergence guarantees of the aforementioned tree-based methods depend substantially on their search strategies. Best-bound search prioritizes nodes that determine the global relaxation bound, whereas depth-first search follows successive descendants to seek feasible incumbents. Best-projection and best-estimate rules incorporate additional information about integer infeasibility and the estimated objective cost of enforcing integrality, favoring nodes expected to contain good feasible solutions. Several hybrid strategies have also been developed to balance these objectives, by combining backtracking and diving heuristics (see \cite{linderoth1999computational} for a survey of all the above).  Although these strategy rules exploit information relevant to feasibility and optimality, their selection criteria are not based on conditions sufficient for progress towards integrality. This motivates a complementary approach in which the node-acceptance mechanism itself provides such analytical control and optimality convergence guarantees.

The aim of this work is to develop a new tree-search framework for \eqref{eq:milp} whose search strategy forces vanishing integer infeasibility and asymptotic progress towards global optimality. Inspired by penalty-free methods in nonconvex constrained optimization \cite{fletcher2002nonlinear}, we develop a \textit{branch-and-filter} (BF) algorithm that enjoys global convergence warranties. The proposed method incorporates a filter mechanism, which compares infeasibility and objective reduction measures to determine the next selected node at each iteration, as to ensure asymptotic feasibility along an infinitely generated fractional sequence. The search strategy includes an additional funnel mechanism, whose primary goal is to keep track of ``promising" (e.g., best-bound) nodes that can drive the iterates to global optimality. When the introduced search strategy cannot successfully identify the next parent node, the algorithm calls a restoration oracle, as it is commonly the case in classical filter methods in the nonconvex optimization literature \cite{FletcherLeyfferToint2002}. Although the restoration oracle can be implemented in practice by modifying standard (BB) or (BC) methods, when the integer variables are bounded, we introduce a novel cell-deletion algorithm that successively restricts the feasibility region by removing cell subfamilies of filter-inadmissible points.

Under standard assumptions, we prove that the (BF) algorithm either terminates finitely with a certificate of infeasibility or an optimal mixed-integer solution, or generates an infinite sequence of iterates whose objective values converge to the optimal one. Importantly, whenever a convergent subsequence exists, for example, when the relaxed linear programming (LP) polyhedral form of \eqref{eq:milp} is compact, it always converges to a global optimal solution. Further, we show that, under a boundedness assumption on the filter objective metric function, (BF) outputs an $\varepsilon$-optimal mixed-integer solution in at most $\mathcal{O}(\varepsilon^{-1})$ outer iterations. This iteration complexity result appears to be novel in the relevant literature, and indicates that the computational hardness of solving \eqref{eq:milp} is explicitly embedded into the introduced restoration phase. Under minor simplifications of the main algorithmic structure, we also show that the filter device can be utilized to determine mixed-integer feasibility of \eqref{eq:milp} instead of global optimality. 

Historically, a couple of works have introduced the terminology of a ``filter" in MILP, but with distinct interpretation and purpose. Balas \cite{balas1967discrete} was the first to introduce a ``filter method" for binary mixed-integer programming, which solves a ``filter-problem" that screens candidate solutions before the main algorithmic tests are applied. San Segundo et al. \cite{san2022new} also use the term “branch-and-filter,” where filtering refers to a process that reduces variable domains through graph-coloring and satisfiability techniques. We are aware of few works that have previously employed similar filter mechanisms to this paper, all of which focus on mixed-integer nonlinear programming (MINLP) problems. Leyffer \cite{leyffer2001integrating} integrates a filter-SQP method with (BB), allowing branching before each continuous nonlinear relaxation is fully solved. This filter method was subsequently combined with a linearization-based (BC) framework \cite{abhishek2010filmint}. In addition, Fernandes et al. \cite{fernandes2014branch} employ a multi-start coordinate-search filter method to solve continuous relaxations within (BB) for nonsmooth, nonconvex MINLP problems. Filters have also been incorporated into pattern-search methods for nonlinearly constrained mixed-variable problems, in a similar derivative-free optimization setting \cite{abramson2004filter}. It is worth mentioning that these approaches use filters with the much different purpose of balancing objective improvement and nonlinear constraint violation, either within mixed-variable search or continuous subproblem solvers.

The rest of the paper is organized as follows. In Section \ref{sec:algorithm} we formally introduce the filter mechanism and present the (BF) algorithm. Sections \ref{sec:theory} and \ref{sec:complexity} include the main global convergence and iteration complexity results, respectively. In Section \ref{sec:restoration} we realize the restoration oracle by introducing a cell-deletion algorithm, which can be solved in finitely many iterations via successive LP solves or disjunctive programming methods, under a bounded integer variable assumption. Section \ref{sec:experiments} investigates the computational performance of (BF) under various implementations of the restoration oracle. Finally, we discuss future directions and natural extensions in Section \ref{sec:discussion}.

\section{Algorithm overview}\label{sec:algorithm}

\subsection{Notation}

We define the objective function $f:\mathbb{R}^{n}\to \mathbb{R}$ with $f(x):=c^\top x$ given the fixed vector $c\in \mathbb{Q}^{n}$, where $n:=p+q\geq 1$, and we write $f_k:=f(x_k)$ for simplicity. For $x=(z,y)\in \mathbb{R}^{p}\times \mathbb{R}^{q}$, we will sometimes write the polyhedral constraint ``$Ax\leq b$" as ``$A_{z} z+ A_{y} y\leq b$" in order to distinguish between the integer and continuous parts, where $A_{z}\in \mathbb{Q}^{m\times p}$ and $A_{y}\in \mathbb{Q}^{m\times q}$. We can thus reformulate \eqref{eq:milp} as $\inf_{x\in X} f(x)$, where
\begin{equation*}
    X:=P\cap (\mathbb{Z}^{p}\times \mathbb{R}^{q})
\end{equation*}
denotes the mixed-integer feasibility set, with $P$ representing the closed polyhedron
\begin{equation*}
    P:=\{(z,y)\in\mathbb{R}^p\times\mathbb{R}^q:A_z z+A_y y\le b\}.
\end{equation*}
We denote by $X^*\subseteq X$ the set of all mixed-integer optimal solutions and by $z^*$ the optimal value of \eqref{eq:milp}, whenever the latter exists and is finite. We further denote $\overline{\mathbb{R}}:=\mathbb{R}\cup \{\infty\}$, $\mathbb{R}_{+}:=[0,+\infty)$, $\mathbf{1}:=(1,...,1)^\top\in \mathbb{R}^{p}$ and $[w]_{+}:=\max\{0,w\}$ for any $w\in \mathbb{R}$. For functions $\zeta:\mathbb{R}\to\mathbb{R}$, $\xi:\mathbb{R}\to[0,\infty)$, we write
$\zeta(\cdot)=\mathcal{O}\big(\xi(\cdot)\big)$ in order to indicate that
$\lvert \zeta(\cdot)\rvert \leq C \xi(\cdot)$ for some $C>0$. Lastly, we denote the cardinality of a set $\mathcal{A}$ by $|\mathcal{A}|$.

\subsection{Branching and incumbent update}

The algorithmic structure resembles that of the state-of-the-art branching iterative methods, by generating a number nodes at each iteration via valid inequality actions. At the start of every outer iteration $k$, (BF) branches an open node $Z_k$, the latter of which is chosen as the filter-admissible node of the previous iteration. The algorithm then solves the relaxed LP formulations of all (finitely many) offspring of $Z_k$. More specifically, a child node $N_k$ represents a polyhedron $P(N_k)\subseteq P$ and the corresponding mixed-integer set $X(N_k):=P(N_k)\cap(\Z^p\times\mathbb{R}^q)$. For every such node, the algorithm computes
\begin{equation}\label{eq:lpoptimal}
    x_{N_k}\in \argmin_{x\in P(N_k)} f(x),
\end{equation}
whenever an optimal solution exists. After the relaxed LPs of all child nodes are solved, the parent $Z_k$ is removed and a standard fathoming procedure follows based on the current incumbent $U_k$. The value $U_k\in\R$ represents the objective value of a stored feasible point, whenever it is finite. Initially, no such point need be available. Given an incumbent $U_k$ and a node $N_k$, we define the \textit{strict improvement sets} by
\begin{equation}\label{eq:strict-target}
 X_{<U_k}:=\begin{cases}X &\;\;\text{if }\;\;U_k=+\infty,\\ \{x\in X:f(x)<U_k\}&\;\;\text{if }\;\;U_k<+\infty,\end{cases}
 \quad\text{and}\quad X_{<U_k}(N_k):=X(N_k)\cap X_{<U_k}.
\end{equation}
During this first phase of iteration $k$, a node $N_k$ can be discarded if its relaxation is empty, if a certified lower bound is at least the current finite incumbent, or if $X_{<U_k}(N)=\varnothing$. If $x_{N_k}=(z_{N_k},y_{N_k})\in X$, the proposed method explicitly uses the following fixed-integer incumbent update scheme: For a verified feasible point $x=(z,y)\in X$ (by omitting the index $N_k$ for simplicity of exposition), we retain its integer vector part $z\in \mathbb{Z}^{p}$ and solve the LP
\begin{equation}\label{eq:update1}
 v(z):=\min_{\widetilde y\in Y(z)}
       c^\top (z,\widetilde y)
\end{equation}
over the continuous fiber
\begin{equation}\label{eq:update2}
    Y(z):=\left\{\widetilde y\in\mathbb{R}^q:A_y\widetilde y\le b-A_z z\right\}.
\end{equation}
We then replace $U_k$ by $\min\{U_k,v(z_{N_k})\}$, and we store a corresponding original-problem feasible point and close the node. Observe that the output solution need not belong to the feasible set of $N_k$, since original-fiber optimization can remove node-specific restrictions on the continuous variables. Whenever a node is closed using a node-optimal solution, that node certificate is retained separately from the possibly improved global incumbent. All subsequent closure arguments use precisely this distinction, and follow from the described incumbent update procedure.

\subsection{The filter mechanism}

In its second phase, the (BF) algorithm incorporates a filter mechanism. This technique is inspired by the one commonly used in nonconvex constrained optimization, whose aim is to replace a penalty function and guarantee convergence to stationary points \cite{FletcherLeyfferToint2002}. The internal goal of a filter device, based on ideas from multi-objective optimization, is to retain pairs of function values that measure objective and infeasibility reduction. In classical nonconvex optimization methods, these functions usually correspond to the objective value and norm of the constraint violation, respectively \cite{fletcher2002nonlinear}. Although our filter mechanism is inspired by this literature and plays the same fundamental role of driving the sequence of iterations to feasibility, the violation measures considered are different. 

Let $x\in\mathbb{R}^{n}$ be a fractional point. First, instead of measuring actual objective progress at $x$, the proposed filter computes the objective improvement that $x$ would offer over the incumbent if it was mixed-integer feasible. In particular, let $\bar{L}:=f_{\text{low}}+\delta$ for some user chosen $\delta>0$, where
$f_{\text{low}}:=\min_{x\in P} f(x)$ denotes the relaxed LP value of \eqref{eq:milp} (if the latter is infeasible, we have nothing to show). We define the \textit{incumbent-objective gap} function $g:\mathbb{R}^{n}\times \R \to \mathbb{R}$ by
\begin{equation}\label{eq:g}
 g(x,U):=\begin{cases}
          \left[\bar L-f(x)\right]_{+}& \;\;\text{if }\;\;U=+\infty,\\
          U-f(x) &\;\;\text{if }\;\;U<+\infty.
         \end{cases}
\end{equation}
Before an incumbent exists, $g$ measures the shortfall below the reference level $\bar L$. Afterward, it simply measures the objective difference from the current incumbent.

The second filter function measures the integer violation of the infeasible point $x=(z,y)\in \mathbb{R}^{n}$. This is slightly different from the state-of-the-art filter methods in nonconvex optimization that study the violation of the entire feasibility set. More precisely, we define the \textit{integrality violation} function $h:\mathbb{R}^{n}\to [0,1/2]$ with
\begin{equation}\label{eq:h}
 h(x):=\dist_\infty(z,\Z^p)
      =\max_{1\le i\le p}\dist(z_i,\Z).
\end{equation}
Although we use the infinity distance norm in this work, others are also possible in theory. The choice \eqref{eq:h} relies on the fact that it produces convenient LP representations of inner subproblems (see Section \ref{sec:restoration}).

We now describe the exact filter mechanism. Let $\F_k:=\{(h_i,g_i):i\in\Ical_k\}$ denote the active \textit{filter}, for some $\mathcal{I}_k\neq \varnothing$, which includes all ``non-dominated" pairs identified by the algorithm at all iterations prior to $k$. Specifically, a pair $(h_i,g_i)$ obtained at iteration $i\in \mathcal{I}_k$ is said to \textit{dominate} another pair $(h(x),g(x,U))$ if
\begin{equation}\label{eq:dominance}
    h_i\leq h(x)\quad \text{and}\quad g_i\leq g(x,U).
\end{equation}
After the first branching phase is complete, the introduced (BF) method runs open nodes (i.e., non-fathomed nodes) through the filter and searches for an admissible point: Let $x_{N_k}\in \mathbb{R}^{n}$ be computed as in \eqref{eq:lpoptimal} for some open node $N_k$. We say that $x_{N_k}$ is \textit{filter-admissible} if 
\begin{equation}\label{eq:admission}
    h(x_{N_k})\leq \beta h_i\quad \text{or}\quad g(x_{N_k},U_k)\leq g(x_i,U_i) - \gamma  h(x_{N_k})\quad\forall i\in \mathcal{I}_k,
\end{equation}
where $\beta,\gamma$ are user-chosen parameters satisfying $\beta,\gamma\in (0,1)$, and $U_k$ is the last update on the incumbent value at iteration $k$. The acceptability criterion \eqref{eq:admission} is close in spirit to that of \cite{chin2003global}. Formally, we introduce the set of admissible points related to the filter $\mathcal{F}_k$ - essentially the non-empty index set $\mathcal{I}_k$ of previous admissible iterates - as follows:
\begin{equation}\label{eq:filter-set}
 \Acal_k^F:=\left\{x\in E(U_k):
 h(x)\le\beta h_i\ \text{or}\
 g(x,U_k)\le g_i-\gamma h(x),\; i\in\Ical_k\right\}.
\end{equation}
Here, $E(U_k)$ represents the set of \textit{eligible points}; we define
\begin{equation}
E(U_k):=
\begin{cases}
\mathbb{R}^{n} & \text{if } U_k=+\infty,\\
\{x\in\mathbb{R}^{n}:f(x)\le U_k\} & \text{if } U_k<+\infty.
\end{cases}
\end{equation}
Note that $g(\cdot,U)$ is always non-negative over $E(U)$. Following the standard filtering procedure of the previous works, if $x_{N_k}\in \mathcal{A}_k^F$ for some open node $N_k$, then all pairs dominated by $(h(x_{N_k}),g(x_{N_k},U_k))$ are removed from the filter (according to \eqref{eq:dominance}), and the filter is updated so that it includes $(h(x_{N_k}),g(x_{N_k},U_k))$. It is important to highlight that, by the first phase of the outer iteration $k$, we always have $h(x_{N_k})>0$ for every open node $N_k$, that is, every filter-admissible point is fractional. To this end, the reason why we choose the filter acceptance test \eqref{eq:admission} instead of simpler ones (e.g., that of \cite{fletcher2002nonlinear}), is because it avoids pathological ``cyclical" scenarios. More specifically, it forbids points from accumulating in local neighborhoods of previously accepted points, thus enforcing $h_k\to 0$. At a high-level, this is tied to the fact that $\mathcal{A}_k^F$ represents a \textit{slanting envelope} (see Figure \ref{fig:filter}).

\begin{figure}[H]
    \centering
 \includegraphics[width=\textwidth]{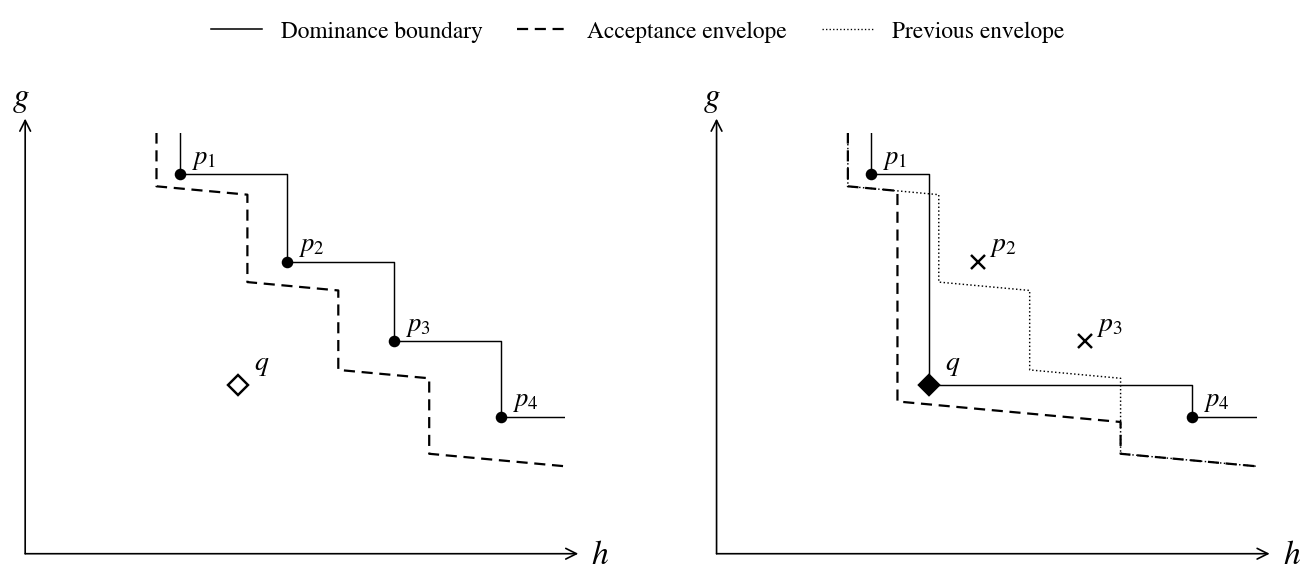}
    \caption{\small A filter admission–deletion update. The point $q$ lies inside the ``slanting envelope" region generated by \eqref{eq:admission}, and is thus filter-admissible. Insertion of $q$ inside the filter tightens the envelope and permits deletion of the dominated entries $p_2$ and $p_3$ according to \eqref{eq:dominance}.}
    \label{fig:filter}
\end{figure}

In strike contrast to classical filter methods that aim on decreasing $f$, decreasing $g$ rewards increasing the candidate objective, which is natural when candidate values represent improving relaxation bounds. This opposite intention can lead to admissions that would be otherwise rejected by a standard $f$-filter (i.e., a filter with metric $f$ in place of $g$). We illustrate this behavior with an example.

\begin{example}\label{ex:ex}
Consider the integer linear program
\begin{equation*}
\min f(z):=z_2\quad \text{s.t.}\;\;z\in P\cap\mathbb Z^{2},\;\; P:=\left\{ z\in [0.1,1]^2:9z_2-2z_1\geq 0.7 \right\}.
\end{equation*}
The unique root LP solution is $z_0=(0.1,0.1)$. Thus, $f(z_0)=0.1$ and $h(z_0)=0.1$. Suppose we perform a branch on $z_1$ using the standard disjunction $z_1\leq \floor{0.1}=0$ and $z_1\geq \ceil{0.1}=1$. The left child node is infeasible. In the right node, say $N$, the existing upper bound forces $z_1=1$, so its relaxation is 
\begin{equation*}
P(N)=\{(1,z_2):0.3\leq z_2\leq 1\}.
\end{equation*}
Thus, its unique LP solution is $z_N=(1,0.3)$, with $f(z_N)=0.3$ and $h(z_N)=0.3$. Choose $\delta=0.5$, $\beta=0.8$, and $\gamma=0.2$. Since the left child is infeasible and the right child's LP solution
is fractional, the branching step does not produce an incumbent,
so $U$ remains $+\infty$. With these parameters, we compute $\bar{L}=0.6$, $g(z_0,+\infty)=0.5$ and $g(z_N,+\infty)=0.3$. The integrality-reduction condition fails because
\begin{equation*}
h(z_N)=0.3>\beta h(z_0)=0.8(0.1)=0.08.
\end{equation*}
 Nevertheless, the reduction condition with respect to $g$ holds:
\begin{equation*}
g(z_N,+\infty)=0.3
\le g(z_0,+\infty)-\gamma h(z_N)
=0.5-0.06=0.44.
\end{equation*}
Thus, $z_N$ is filter-admissible. By contrast, the corresponding $f$-filter is initialized with the pair $(0.1,0.1)$, and rejects the same candidate $z_N$. Indeed, the integrality condition still fails, and
\begin{equation*}
f(z_N)=0.3
>
f(z_0)-\gamma h(z_N)
=0.1-0.06=0.04.
\end{equation*}
Both coordinates $(h,f)$ increase, whereas the pair $(h,g)$
trades increased integrality violation for a decreased shortfall
below the reference level $\bar{L}$. Since the left child node is closed, all integer-feasible solutions remain in $P(N)$,
and the global relaxation bound increases. Thus, the filter with the introduced incumbent-objective gap metric accepts the LP solution associated with a stronger lower bound, despite its increased objective value
and integrality violation, whereas the corresponding $f$-filter
rejects it.

\end{example}

\subsection{The funnel mechanism}

In the case when the proposed algorithm does not output a certificate of infeasibility or an optimal solution in finitely many iterations, the filter mechanism is responsible for driving the iterates to mixed-integer feasibility. However, it cannot guarantee convergence to optimality on its own. In the context of nonconvex constrained optimization, optimality is commonly achieved by implementing the filter method in line-search or trust-region frameworks \cite{fletcher2002global,wachter2005line}, and by classifying iterates to different types, depending on the reduction given by certain model functions. In this paper, we utilize a funnel mechanism, which has been known to guarantee convergence to optimality, under a different optimization context \cite{gould2010nonlinear}. 

Formally, we introduce an arbitrary (user-chosen) continuous and non-decreasing \textit{funnel function} of the form
\begin{equation}\label{eq:funnelfunction}
    \omega:[0,1/2]\to\mathbb{R}_+,\quad \text{with}\;\; \omega(0):=0.
\end{equation}
 After the branching procedure, incumbent updates and fathoming tests have been performed, all nodes $N_k$ that are still open are added to the \textit{open pool} $\Ocal_k$, and their relaxed LP optimal solutions are stored. We then compute the global relaxation lower bound over the set of all open nodes
\begin{equation}\label{eq:low}
    \ell_k:=\min_{N\in \Ocal_k}f(x_{N})
\end{equation}
with $x_N$ as in \eqref{eq:lpoptimal}. At a non-empty open pool state, we define the \textit{funnel set}
\begin{equation}\label{eq:funnel}
 \Gcal_k:=\left\{x\in\bigcup_{N\in\Ocal_k}P(N):
                 f(x)-\ell_k\le\omega(h(x))\right\}.
\end{equation}
At its core, (BF) searches for a filter-admissible point $x_{N}$ over $\mathcal{G}_k$, i.e., for a \textit{filter-funnel-admissible} candidate. This is the key mechanism that guarantees both asymptotic feasibility and asymptotic optimality for the class of MILP problem considered (Theorem \ref{th:main}).

\subsection{The restoration phase}

Despite the fact that a funnel-admissible point always exists (by \eqref{eq:low}), it might be the case that $x_N\not\in \mathcal{A}_k^F$ for every $N\in \mathcal{O}_k$. As it is common in the literature of filter methods, the algorithm resorts to a \textit{restoration phase}\footnote{This is oftentimes called \textit{feasibility restoration} \cite{nocedal2006numerical}, since its main goal is to compute a ``sufficiently feasible" solution. We deviate from this terminology due to the fact that our restoration phase serves a different purpose.} that prohibits the algorithm from stalling. The restoration oracle employed here is, however, substantially different from those seen in nonconvex optimization, in the sense that it is searching for a filter-admissible point that is not a minimizer of the form \eqref{eq:lpoptimal}, instead of an approximate feasible solution. To be more specific, the restoration phase calls an oracle on an open node $N_k$ and returns one of the following (ordered) outcomes:
\begin{description}\hypertarget{or1}{}\hypertarget{or2}{}\hypertarget{or3}{}
\item[\hyperlink{or1}{(OR1)}] A certificate that $X_{<U_k}(N_k)=\varnothing$.
\item[\hyperlink{or1}{(OR2)}] A point $x_{N_k}^*\in\argmin_{x\in X(N_k)}f(x)$ with $f(x_{N_k}^*)<U_k$.
\item[\hyperlink{or1}{(OR3)}] A point $x_{k+1}\in\Acal_k^{\res}(N_k)\setminus X(N_k)$, where $\mathcal{A}_k^{\text{res}}(N_k):=P(N_k)\cap \mathcal{A}_k^F\cap \mathcal{G}_k$.
\end{description}

If \hyperlink{or1}{(OR1)} or \hyperlink{or1}{(OR2)} is returned by the oracle, then the incumbent is updated (by the second outcome) and the node is fathomed. The restoration phase then moves to another open node, and repeats the same procedure, until a filter-admissible iteration is returned \hyperlink{or1}{(OR3)}, or the open pool $\mathcal{O}_k$ becomes empty. In the latter case, the algorithm terminates with either a certificate of infeasibility or a global optimal solution, depending on the value of $U_k$.

If \hyperlink{or1}{(OR3)} is returned by the oracle, then the pair $(h,g)$ corresponding to $x_{k+1}$ is added to the filter and all pairs dominated by the latter are removed. In the case that the restoration oracle is implemented by a cutting procedure that modifies the geometric structure of $P(N_k)$ (as is the case in the proposed realization in Section \ref{sec:restoration}), the algorithm must replace the node $N_k$. Specifically, (BF) chooses a replacement node $Z$ such that $P(Z)$ is a polyhedron and
\begin{equation}\label{eq:replacement}
    x_{k+1}\in P(Z)\subseteq P(N_k),
    \qquad
    \bigl(X(N_k)\setminus X(Z)\bigr)\cap X_{<U_k}
    =\varnothing.
\end{equation}
Note that such a replacement always exists, since one can merely choose $Z$ with $P(Z)=P(N_k)$. Thus, \eqref{eq:replacement} imposes no additional existence assumption on the restoration oracle.

\subsection{The branch-and-filter algorithm}

The algorithm just described is given below. An initialization procedure solves the LP relaxation of the original MILP problem \eqref{eq:milp} and returns a solution $x_0\in\mathbb{R}^{n}$. If $x_0\in X$, then the algorithm terminates with an optimal mixed-integer solution. Otherwise, we set $f_{\text{low}}$ to be the optimal value, and initialize $U_0:=+\infty$. Further, right after the first branching phase, and right before invoking the oracle in the restoration phase, our algorithm performs a series of global tests: An empty pool $\mathcal{O}_k$ returns infeasibility if $U_k=+\infty$ and optimality otherwise. For a nonempty pool, the additional test $U_k<+\infty$ and $U_k\le\ell_k$ returns global optimality. Finally, at the end of every outer iteration $k$ the filter admits a point $x_{k+1}$ and we set $g_{k+1}:=g(x_{k+1},U_k)$ and $h_{k+1}:=h(x_{k+1})$ for simplicity, where $U_k$ represents the stored incumbent value at admission.

\begin{algorithm}[!htbp]
\caption{Branch-and-filter (BF)}\label{alg:bf}
\small
\begin{algorithmic}[]
\Require Problem \eqref{eq:milp} parameters; $\beta\in (0,1)$, $\gamma\in (0,1)$, $\delta>0$; funnel function $\omega$.
\State Solve the LP $\min_{x\in P}f(x)$.
\If{$P=\varnothing$} 
\State \textbf{terminate} with $X=\varnothing$. 
\EndIf
\State Choose $x_0\in \argmin_{x\in P}f(x)$.
\If{$x_0\in X$}
\State \textbf{terminate} with $x_0\in X^*$.
\EndIf
\State Set $N_0:=N^{\text{root}}$, $f_{\text{low}}:=f(x_0)$, $\bar L:=f_{\text{low}}+\delta$, $U_0:=+\infty$, $\Ocal_0:=\{N_0\}$, $Z_0:=N_{0}$, and $\F_0:=\{(h(x_0),\delta)\}$, where $N^{\text{root}}$ represents the root node of the branching tree.
\For{$k=0,1,2,\ldots$}
    \State Perform a branch action at $Z_k$ and replace it in $\Ocal_k$ with its successor family $\{N_{k,1},...,N_{k,s}\}$.
    \For{$j=1,...,s$}
    \State Compute $x_{N_{k,j}}$ by \eqref{eq:lpoptimal} (set $\Ocal_k\gets\Ocal_k\setminus\{N_{k,j}\}$ whenever $P(N_{k,j})=\varnothing$, and move to $j+1$).
    \If{$x_{N_{k,j}}\in X$}
    \State Update $U_k$ by the procedure in \eqref{eq:update1}-\eqref{eq:update2} and set $\Ocal_k\gets\Ocal_k\setminus\{N_{k,j}\}$.
    \EndIf
    \EndFor
    \State Fathom all nodes whose optimal relaxed LP value is at least $U_k$.
    \If{$\Ocal_k=\varnothing$}
    \State \textbf{terminate} with $X=\varnothing$ if $U_k=+\infty$, and with $x^*_k\in X$ such that $f(x_k^*)=U_k$ otherwise.
  \EndIf
  \State Compute $\ell_k$  by \eqref{eq:low}.
  \If{$U_k<+\infty$ and $U_k\le\ell_k$}
    \State \Return $U_k=z^*$ and \textbf{terminate} with $x^*_k\in X$ such that $f(x_k^*)=U_k$.
  \EndIf
  \If{$x_{N}\in \mathcal{A}_k^F\cap \mathcal{G}_k$ for some $N\in \mathcal{O}_k$}
    \State Set $(x_{k+1},N_k^*):=(x_N,N)$ and $d_k:=\textnormal{FP}$.
  \Else
    \State Go to \textbf{Restoration Phase} (Algorithm \ref{alg:rp}).
    \State Receive $(x_{k+1},N_k)$, updated $U_k$, $\Ocal_k$ and its (optional) lower-bound certificate.
    \State Recompute $\ell_k$ by \eqref{eq:low} and set $d_k:=\textnormal{RP}$.
  \EndIf
  \State Set $(h_{k+1},g_{k+1}):=(h(x_{k+1}),g(x_{k+1},U_k))$, where $U_k$ represents the last incumbent update.
  \State Delete coordinate-wise dominated pairs by \eqref{eq:dominance} and insert $(h_{k+1},g_{k+1})$ into $\F_k$.
  \If{$d_k=\textnormal{RP}$}
    \State Replace $N_k$ by $Z$ in $\Ocal_k$ according to \eqref{eq:replacement} and store its certified LP value.
  \Else
    \State Set $Z:=N_k^*$.
  \EndIf
  \State Set $\mathcal{O}_{k+1}\gets \Ocal_k$, $\mathcal{F}_{k+1}\gets \mathcal{F}_k$, $U_{k+1}\gets U_k$ and $Z_{k+1}\gets Z$.
\EndFor
\end{algorithmic}
\end{algorithm}

\begin{algorithm}[!htbp]
\caption{Restoration Phase}\label{alg:rp}
\small
\begin{algorithmic}[]
\Require Outer iteration $k$, open pool $\Ocal_k$, incumbent $U_k$, filter $\mathcal{F}_k$.
  \If{$\Ocal_k=\varnothing$}
    \State \textbf{terminate} with $X=\varnothing$ if $U_k=+\infty$, and with $x^*_k\in X$ such that $f(x_k^*)=U_k$ otherwise.
  \EndIf
  \State Recompute $\ell_k$  by \eqref{eq:low}.
  \If{$U_k<+\infty$ and $U_k\le\ell_k$}
    \State \Return $U_k=z^*$ and \textbf{terminate} all algorithms with $x^*_k\in X$ such that $f(x_k^*)=U_k$.
  \EndIf
  \State Choose $N_k\in\Ocal_k$ and call the \textbf{Restoration Oracle} with input $\Ocal_k$, $U_k$, $\mathcal{F}_k$, and $N_k$.
  \If{the outcome is \hyperlink{or1}{(OR1)}}
    \State Set $\Ocal_k\gets\Ocal_k\setminus\{N_k\}$ and \Return to \textbf{Restoration Phase} (Algorithm \ref{alg:rp}).
  \ElsIf{the outcome is \hyperlink{or1}{(OR2)}}
    \State Update $U_k$ by the procedure in \eqref{eq:update1}-\eqref{eq:update2}, set $\Ocal_k\gets\Ocal_k\setminus\{N_k\}$ and \Return to \textbf{Restoration Phase} (Algorithm \ref{alg:rp}).
  \Else
    \State \Return $(x_{k+1},N_k)$ to \textbf{Branch-and-filter} (Algorithm~\ref{alg:bf}).
  \EndIf
\Statex \hspace{-\algorithmicindent}\underline{\textbf{Restoration Oracle}}
\State \textbf{Input:} Open pool $\mathcal{O}_k$, node $N_k\in \mathcal{O}_k$, incumbent $U_k$, filter $\mathcal{F}_k$.
\State \textbf{Output:} Exactly one of \hyperlink{or1}{(OR1)}-\hyperlink{or1}{(OR3)}.
\end{algorithmic}
\end{algorithm}

\section{Global convergence and optimality}\label{sec:theory}

We make the following assumptions throughout the paper:\hypertarget{f1}{}\hypertarget{f2}{}\hypertarget{f3}{}
\vspace{-0.1cm}
\begin{description}
\item[\hyperlink{f1}{(F1)}] If $P\ne\varnothing$, then $\min_{x\in P}f(x)=f_{\text{low}}>-\infty$.
\vspace{-0.1cm}
\item[\hyperlink{f2}{(F2)}] All LP subproblems considered are solved exactly. Further, all branching procedures are finite and valid: A branch action on a node $Z_k$ produces child nodes $N_{k,1},...,N_{k,s}$, for some $s\geq 1$, which satisfy $P(N_{k,j})\subseteq P(Z_k)$ for every $j$ and $X(Z_k)=X(N_{k,1})\cup ... \cup X(N_{k,s})$. 

\item[\hyperlink{f3}{(F3)}] The restoration oracle returns exactly one of \hyperlink{or1}{(OR1)}-\hyperlink{or1}{(OR3)} after finitely many iterations.
\end{description}

Boundedness of the optimal value (whenever one exists) is a common assumption in the relevant literature (for instance, see \cite{boland2015augmented,feizollahi2017exact}). \hyperlink{f1}{(F1)} is also strictly weaker than compactness of a nonempty relaxation \cite{sun2024decomposition, testa2019distributed}. Since a linear function with finite infimum on a nonempty polyhedron attains that infimum, \hyperlink{f1}{(F1)} provides the root optimizer and finite bound $f_{\text{low}}$; the same conclusion holds for every nonempty node relaxation contained in $P$ \citep{Schrijver1986}.
The finite-action requirement in \hyperlink{f2}{(F2)} excludes an unending branching loop hidden inside a single action. Standard proper integer branching satisfies the prescribed coverage requirements \cite{linderoth1999computational}.  Lastly, \hyperlink{f3}{(F3)} is a substantial oracle condition whose finite step count is generally hard to satisfy in practice, unless additional assumptions are imposed. In Section \ref{sec:restoration}, we will provide a novel realization under a boundedness assumption on the integer variables of \eqref{eq:milp}.

The filter mechanism implicitly requires information about all admitted pairs, although admission tests use only the current filter. This will be made clear in the proof of the main result. The following lemma justifies the passage from current to historical records.

\begin{lemma}\label{lem:historical}
Suppose \hyperlink{f1}{(F1)}-\hyperlink{f3}{(F3)} hold. Then, for every $i<j$, we have $h_j\le\beta h_i$ or $g_j\le g_i-\gamma h_j$.
\end{lemma}
\begin{proof}
Consider the filter immediately before the acceptance test for the admitted iterate $j$. The candidate record $j$ has not been inserted yet. We may assume, without any loss of generality, that $(h_i,g_i)\not\in\mathcal{F}_{j-1}$ (otherwise we have nothing to show). In other words, iteration $i$ was removed form the filter at an earlier iteration before $j$. Let $i_1<...<i_n$ be the chain of deleted iterates between $i$ and $j$, meaning that $i_{l}$ was deleted due to dominance of $i_{l+1}$, with $i<i_1$ being the first, and $i_n<j$ being the last one, which is also the only active iterate in this chain, for some $n\geq 1$. By the deletion criterion \eqref{eq:dominance}, we have
\begin{equation}\label{eq:chain1}
    h_i\geq h_{i_1}\geq h_{i_2}\geq ... \geq h_{i_n}\quad \text{and}\quad g_i\geq g_{i_1}\geq g_{i_2}\geq ... \geq g_{i_n}.
\end{equation}
Since $(h_{i_n},g_{i_n})\in \mathcal{F}_{j-1}$ and $(h_j,g_j)$ is accepted by the filter mechanism, by definition, we have 
\begin{equation}\label{eq:chain2}
    h_{j}\leq \beta h_{i_{n}}\quad \text{or}\quad g_j\leq g_{i_{n}}-\gamma h_j.
\end{equation}
The result directly follows from \eqref{eq:chain1} and \eqref{eq:chain2}.
\end{proof}

We next prove the validity of the prescribed node updates, namely we show that every mixed-integer feasible point in the processed node that strictly improves the updated incumbent remains covered by the replacement nodes.

\begin{lemma}\label{lem:replacement}
Suppose that \hyperlink{f1}{(F1)}-\hyperlink{f3}{(F3)} hold and that a prescribed action at iteration $k$ processes a node $N_k$ and produces an incumbent update $U_k^+\leq U_k$ and an open pool update $\Ocal_k^{+}$. Then, there exists a finite (possibly empty) family of nodes $\mathcal S_k$ representing that update such that
\begin{equation}\label{eq:family}
 X(N_k)\cap X_{<U_k^+}\subseteq\bigcup_{M\in\mathcal S_k}X(M)
\end{equation}
and
\begin{equation}\label{eq:pool}
     \Ocal_k^{+}=(\Ocal_k\setminus\{N_k\})\cup\mathcal S_k.
\end{equation}
\end{lemma}

\begin{proof}
    We take cases for all possible nodes to which the action is prescribed at outer iteration $k$, and we construct a family for each one:

    \underline{\textit{1st case}}: The first case has to do with an initial action on $N_k=Z_k$. In this case, the algorithm performs a valid branch process on $N_k$ with child nodes $N_{k,1},...,N_{k,s}$ for some $s\geq 1$. Then, the incumbent is not updated (so $U_k^{+}=U_k$) and \eqref{eq:pool} holds for $\mathcal{S}_k:=\{N_{k,1},...,N_{k,s}\}$, since $N_k$ is removed after the branching procedure. The inclusion \eqref{eq:family} holds due to \hyperlink{f2}{(F2)}. 

    \underline{\textit{2nd case}}: After a branch action has been applied to $Z_k$, the algorithm next solves the relaxed LPs of the child nodes generated from the latter action. So, in this case, let $N_k$ simply be one of these child nodes. We take various subcases, depending on the outcome of this procedure:
    \begin{itemize}
        \item[--] Suppose that $P(N_k)=\varnothing$. This scenario is trivial, since $P(N_k)=\varnothing$ directly implies $X(N_k)=\varnothing$. The algorithm immediately fathoms the node and removes it from the open pool, so \eqref{eq:family}-\eqref{eq:pool} hold with $U_k^{+}=U_k$ and $\mathcal{S}_k=\varnothing$.
        
        \item[--] Suppose that $f(x_{N_k})\geq U_k$. In that case we have $f(x)\geq f(x_{N_k})\geq U_k\geq U_k^{+}$ for every $x\in X(N_k)$, so $X(N_k)\cap X_{<U_k^+}=\varnothing$. The node $N_k$ is fathomed, and \eqref{eq:family}-\eqref{eq:pool} hold for an empty family $\mathcal{S}_k$.

        \item[--] Suppose that $f(x_{N_k})<U_k$ and $x_{N_k}\not\in X(N_k)$. The incumbent cannot be updated since the optimal solution of the relaxed LP is not integral, so $U_k^{+}=U_k$. Still, $N_k$ remains in the open pool $\mathcal{O}_k$, i.e., the desired properties hold with $\mathcal{S}_k=\{N_k\}$.

        \item[--] If $f(x_{N_k})<U_k$ and $x_{N_k}\in X(N_k)$, then the incumbent is updated as in \eqref{eq:update1}-\eqref{eq:update2}. In particular, we have $f(x)\geq f(x_{N_k})\geq v(z_{N_k})=U_k^{+}$ for every $x\in X(N_k)$. Hence, no point in $X(N_k)$ can strictly improve $U_k^+$, and we may close the node with $\mathcal S_k=\varnothing$. 
    \end{itemize}

    \underline{\textit{3rd case}} The final case reduces to the final incumbent and pool update after the restoration phase has terminated. Outcome \hyperlink{or2}{(OR2)} is handled analogously to the previous case above. For outcome \hyperlink{or1}{(OR1)}, its certificate states $X(N_k)\cap X_{<U_k}=\varnothing$. Since $U_k^+\le U_k$, the set $X_{<U_k^+}$ is a subset of $X_{<U_k}$, including when the earlier incumbent is infinite. Consequently the empty family is valid here as well. If a filter-admissible point is returned instead, i.e., the outcome is \hyperlink{or3}{(OR3)}, then the replacement condition \eqref{eq:replacement} directly gives the desired properties with $\mathcal{S}_k=\{Z\}$.
\end{proof}

Our next task is to show that every globally feasible point that can improve the incumbent is still represented by at least one open node.

\begin{lemma}\label{lem:coverage}
    Suppose \hyperlink{f1}{(F1)}-\hyperlink{f3}{(F3)} hold. Then, for every $k\geq 0$, we have
    \begin{equation}\label{eq:cont}
        X_{<U_k}\subseteq \bigcup_{N\in \mathcal{O}_k}X(N).
    \end{equation}
\end{lemma}

\begin{proof}
    We proceed by induction. The incumbent $U_k$ and open pool $\Ocal_k$ represent their final updated forms during iteration $k$, right before acceptance of $x_{k+1}$ to the filter. At initialization, the root node covers $X$, so the desired inclusion holds. Each subsequent node update preserves this inclusion by Lemma \ref{lem:replacement} and monotonicity of the incumbent. Applying this observation successively to the updates preceding the first admission establishes \eqref{eq:cont} for $k=0$. The same inclusion also holds if a global termination test is reached before that admission. Next, suppose \eqref{eq:cont} is true for $k\geq 0$, and let $x\in X_{<U_{k+1}}$. Monotonicity of the sequence of incumbent parameters gives $X_{<U_{k+1}}\subseteq X_{<U_k}$, so $x\in X(N)$ for some $N\in \mathcal{O}_k$ by the inductive hypothesis. We may assume without loss of generality that $N\not\in \mathcal{O}_{k+1}$. This means that a prescribed node update resulted in removal of $N$ from the open pool, an update $U_{k+1}\leq U_k^{+}\leq U_k$ on the incumbent, a construction of a family $\mathcal{S}_k$ and an update $\mathcal{O}_k^{+}$ on the pool $\mathcal{O}_k$, so that
    \begin{equation}\label{eq:inc}
        x\in X(N)\cap X_{<U_{k+1}} \subseteq X(N)\cap X_{<U_k^{+}}\subseteq \bigcup_{M\in\mathcal S_k}X(M)
    \end{equation}
   by Lemma 3.2 and monotonicity of the incumbent sequence. From \eqref{eq:pool} and \eqref{eq:inc} we deduce that $x\in X(M)$ for some $M\in \mathcal{O}_k^{+}$. If $M\in \mathcal{O}_{k+1}$ we have nothing to show, thus we assume that $\mathcal{O}_k^{+}$ does not represent the final pool update and that $M\not\in \mathcal{O}_{k+1}$. Since only finitely many nodes can be removed, and the open pool can be updated only a finite number of times between the two snapchots, the same argument applied repeatedly (invoking Lemma \ref{lem:replacement}) implies that the point $x$ will eventually belong to the mixed-integer feasibility set of some node member of $\mathcal{O}_{k+1}$. This completes the induction.  
\end{proof}

Next, we show that the funnel set is always well-defined and that the best-bound relaxed node has a local optimal value controlled by the global optimal one.

\begin{lemma} \label{lem:ellbound}
Suppose \hyperlink{f1}{(F1)}-\hyperlink{f3}{(F3)} hold and let an iterate $k\geq 0$ that does not terminate Algorithm \ref{alg:bf}. Then, $f_{\text{low}}\le\ell_k\le f(x)$ for every $x\in\bigcup_{N\in\Ocal_k}P(N)$. In addition, if $X\neq \varnothing$, then $\ell_k\leq z^*$.
\end{lemma}

\begin{proof}
Since (BF) does not terminate at iteration $k$, we must have $\Ocal_k\neq \varnothing$. For the first result, select any $x\in P(N)$ with $N\in\Ocal_k$. Then,
\begin{align}\label{eq:relax}
 f_{\text{low}}\le\ell_k
 &=\min_{M\in\Ocal_k}f(x_M)
 \le f(x_N)
 \le f(x).
\end{align}
If $x$ is a convex combination of points in the union, then \eqref{eq:relax} still holds due to linearity of $f$.

Now, assume that $X\ne\varnothing$. In that case, $z^*$ is well-defined and finite, since $f_{\text{low}}$ is a lower bound and any feasible point gives a finite upper bound. A finite incumbent is feasible, so $U_k\ge z^*$. First suppose $U_k>z^*$, allowing $U_k=+\infty$. For every sufficiently small $\epsilon>0$, there exists $x_\epsilon\in X$ with $f(x_\epsilon)<z^*+\epsilon\leq U_k$. Lemma 3.3 yields $x_\epsilon\in X(N)$ for some open node $N\in\mathcal{O}_k$, so the first part \eqref{eq:relax} gives $\ell_k\le f(x_\epsilon)<z^*+\epsilon$. Since $\epsilon$ was arbitrary, we get $\ell_k\leq z^*$. The remaining case is $U_k=z^*<+\infty$.
Due to the fact that (BF) does not terminate at $k$, all global tests must fail during that outer iteration, meaning that $U_k>\ell_k$. The result follows.
\end{proof}

\begin{remark}\label{rem:tests}
The global tests can now be rigorously justified. An empty pool and Lemma \ref{lem:coverage} imply $X=\varnothing$ when $U_k=+\infty$, or the absence of a point better than the feasible incumbent when $U_k$ is finite, which leads to termination with an optimal solution. For a nonempty pool with $U_k\le\ell_k$, a point $x\in X$ with $f(x)<U_k$ would belong to an open relaxation by Lemma \ref{lem:coverage} and would contradict \eqref{eq:relax}. Consequently, this test also proves global optimality.
\end{remark}

We are finally ready to present the main result of this work.

\begin{theorem}\label{th:main}
Suppose that \hyperlink{f1}{(F1)}-\hyperlink{f3}{(F3)} hold. Then, Algorithm~\ref{alg:bf} has exactly one of the following outcomes:
\begin{enumerate}
\item finite termination with a valid certificate that $X=\varnothing$;
\item finite termination with a mixed-integer optimal solution $x^*\in X^*$;
\item an infinite sequence $\{x_k\}\subseteq P$ with $h(x_k)\to 0$ and $f(x_k)\to z^*$.
\end{enumerate}
Moreover, in the third case, $\dist_2(x_k,X)\to0$ and every limit point of $\{x_k\}$ lies in $X^*$.
\end{theorem}

\begin{proof}
We start by examining the possible outcomes of the initialization phase. If the original relaxation LP gives $P=\varnothing$, then the algorithm terminates with a valid certificate of infeasibility. If, on the other hand, we get a solution $x_0\in\argmin_{x\in P}f(x)$ such that $x_0\in X$, then the algorithm terminates with $x_0\in X^*$. Otherwise, root branching sets the start of the algorithm.

Next, suppose that there exists $k^*\geq 1$ such that $\mathcal{O}_{k^*}=\varnothing$. An infinite incumbent directly indicates global infeasibility. Recall that the global tests run are valid (Remark \ref{rem:tests}), so $U_{k^*}<+\infty$ yields a global optimality certificate. The mixed-integer solution $x^*$ such that $f(x^*)=U_{k^*}$, the latter of which is always computed by construction, must belong to $X^*$.

If the open pool remains non-empty at the end of every single iteration, meaning that all previously conducted optimality tests have failed, then the algorithm generates an infinite sequence $\{x_k\}$ in the relaxed polyhedron $P$. Suppose, by contradiction, that $h_k\not\to0$. Then, there exists $\varepsilon>0$ such that $h_k\geq \varepsilon$ for infinitely many iterates $k\geq 0$. By compactness of $[\varepsilon,1/2]$, we can find a subsequence $\{h_{k_j}\}$ of $\{h_k\}$ such that
\begin{equation}
    h_{k_j}\longrightarrow\bar h
\end{equation}
for some $\bar{h}\in [\varepsilon,1/2]$.
Because $\beta<1$, we have
\begin{equation}
    h_{k_{j+1}}-\beta h_{k_j}\to(1-\beta)\bar h>0.
\end{equation}
Thus, there exists $j_0\geq 0$ such that
$h_{k_{j+1}}>\beta h_{k_j}$ and
$h_{k_{j+1}}\ge\bar h/2$ for every $j\geq j_0$.
Lemma \ref{lem:historical} yields
\begin{equation}\label{eq:tech1}
    g_{k_{j+1}}
    \le g_{k_j}-\gamma h_{k_{j+1}}
    \le g_{k_j}-\frac{\gamma\bar h}{2}\quad \forall j\geq j_0.
\end{equation}
In turn, for any fixed $J>j_0$, \eqref{eq:tech1} gives
\begin{align}
    (J-j_0)\frac{\gamma\bar h}{2}
    &\le
    \sum_{j=j_0}^{J-1}
    \bigl(g_{k_j}-g_{k_{j+1}}\bigr) \notag\\
    &=g_{k_{j_0}}-g_{k_J} \nonumber \\
    &\le g_{k_{j_0}},\label{eq:tech2}
\end{align}
where we used non-negativity of $g$ over the space of all filter-admissible iterates. Letting $J\to\infty$ in \eqref{eq:tech2} contradicts finiteness of $g_{k_{j_0}}$.
Consequently, integrality violation converges to zero.

Next, we prove global optimality. In order to guarantee that the objective values will converge to $z^*$, we must first show that the problem \eqref{eq:milp} is consistent so that $z^*$ is well-defined. To this end, let $Q:=\proj_z P$ be the projection of the relaxed polyhedron on the space of integer variables. Projection of a rational polyhedron is also a rational polyhedron (e.g., see \citep{Schrijver1986}). Therefore, we can write $Q=\{z\in \mathbb{R}^{p}:Bz\le d\}$ for some $B\in \mathbb{Q}^{s\times p}$, $b\in \mathbb{Q}^{s}$, $s\geq 0$. In fact, by removing redundant zero rows, and scaling the remaining finitely many inequalities, we may assume that $B$ and $d$ are integral. We define the integer-recovery measure
\begin{equation}\label{eq:measure}
    \mu_Q:=\inf\left\{ \text{dist}_{\infty}(m,Q):m\in \mathbb{Z}^p\setminus Q \right\}.
\end{equation}
 (If $Q=\mathbb{R}^p$, we simply set $\mu_Q:=+\infty$.) Since $B$, $d$ have integer entries, given $m\in\Z^p\setminus Q$, we can find some row $i$ that satisfies $B_i m-d_i\ge1$. For every $z\in Q$, H\"older's inequality gives
\begin{align}\label{eq:lattice-separation}
 1\le B_i m-d_i\le B_i(m-z)
 \le\norm{B_i}_1\norm{m-z}_\infty.
\end{align}
Since $m\in \mathbb{Z}^{p}\setminus Q$ was arbitrary, we obtain
\begin{equation}\label{eq:rho}
 \mu_Q\geq \min_i\frac1{\norm{B_i}_1}>0.
\end{equation}
Therefore, we must have $h_k<\mu_Q$ for all $k$ sufficiently large. Let us fix an iterate $k$ sufficiently large, and write $x_k=(z_k,y_k)\in \mathbb{R}^{p}\times \mathbb{R}^{q}$. In turn, choose a nearest integer vector $m_k$ to $z_k$. Since $z_k\in Q$, the inequalities $h(x_k)<\mu_Q$ and $m_k\notin Q$ would imply
\begin{equation}\label{eq:round-contradiction}
 \mu_Q\le\dist_\infty(m_k,Q)\le\norm{m_k-z_k}_\infty=h(x_k)<\mu_Q,
\end{equation}
a contradiction. Thus, $m_k\in Q$ for every sufficiently large $k$. Equivalently, $Y(m_k)\ne\varnothing$ (see \eqref{eq:update2}), which directly implies $X\ne\varnothing$ and $f_{\text{low}}\le z^*<+\infty$.

Since $X\neq \varnothing$ and the algorithm generates an infinite sequence, Lemma \ref{lem:ellbound} gives $\ell_k\leq z^*$ for every $k\geq 0$. Further, because $x_{k+1}\in \mathcal{G}_k$ at admission, we have
\begin{equation}\label{eq:funnelproperty}
    f_{k+1}-\ell_k\leq \omega(h_{k+1})
\end{equation}
at the end of every iteration $k\geq 0$. Therefore,
\begin{equation}\label{eq:tech3}
    f_{k+1}\leq \ell_k+\omega(h_{k+1})\leq z^*+\omega(h_{k+1}).
\end{equation}
As $h_k\to 0$, continuity of the funnel function $\omega$ and \eqref{eq:tech3} imply
\begin{equation}\label{eq:limsup}
    \lim\sup_{k}f_k\leq z^*.
\end{equation}
Next, we apply Hoffman's error bound \citep{Hoffman1952} to the fixed matrix $A_y\in \mathbb{Q}^{m\times q}$. More specifically, there exists a constant $C_y<+\infty$, depending on $A_y$ (but independent of $k$), such that
\begin{equation}\label{eq:hoffman}
 \dist_2(y_k,Y(m_k))
 \le C_y\norm{\pos{A_y y_k-b+A_zm_k}}_\infty
\end{equation}
for all $k$ sufficiently large. Since $A_z z_k+A_y y_k\le b$, H\"older's inequality gives
\begin{align}\label{eq:holder}
 \norm{\pos{A_y y_k-b+A_zm_k}}_\infty
 \le\norm{A_z(m_k-z_k)}_\infty
 \le\norm{A_z}_{\infty}h(x_k).
\end{align}
Since the continuous fiber $Y(m_k)$ is a nonempty closed polyhedron for all $k$ sufficiently large, the continuous component $y_k$ of $x_k$ has a nearest Euclidean point $\widetilde y_k\in Y(m_k)$. Hence, $\widetilde x_k:=(m_k,\widetilde y_k)\in X$ for all $k$ sufficiently large, and \eqref{eq:hoffman}-\eqref{eq:holder} yield
\begin{equation}\label{eq:repair-constant}
 \norm{\widetilde y_k-y_k}_2\le C_y\norm{A_z}_{\infty}h(x_k).
\end{equation}
Feasibility of $\widetilde{x}_k$ and linearity of the objective function imply
\begin{align}\label{eq:tech4}
 z^*&\le f(\widetilde x_k) \nonumber \\
 &= f(x_k)+c_z^\top(m_k-z_k)+c_y^\top(\widetilde y_k-y_k)\notag\\
 &\le f(x_k)+C^* h(x_k),
\end{align}
where
\begin{equation}\label{eq:constant}
  C^*:=\norm{c_z}_1+C_y\norm{c_y}_2\norm{A_z}_{\infty}<+\infty.
\end{equation}
Asymptotic feasibility along with \eqref{eq:tech4} yield
\begin{equation}\label{eq:liminf}
    z^*\leq \lim \inf_{k}f(x_k).
\end{equation}
By combining \eqref{eq:limsup} and \eqref{eq:liminf} we finally obtain $f_k\to z^*$.

Moving to the final part, observe that the nearby feasible points above satisfy
\begin{equation}\label{eq:distance-X}
 \dist_2(x_k,X)\le\norm{x_k-\widetilde x_k}_2
 \le\sqrt{p+C_y^2\norm{A_z}_{\infty}^2}\,h(x_k)\longrightarrow0.
\end{equation}
Lastly, if a subsequence $\{x_{k_j}\}$ of $\{x_k\}$ converges to some point $x^*$, then closedness of $P$ and continuity of $h$ imply $x^*\in X$. Continuity of $f$ gives $f(x^*)=z^*$, so $x^*\in X^*$. 
\end{proof}

\begin{remark}\label{rem:f1}
    It is noteworthy that, although the filter is utilizing the incumbent gap function $g$, the objective function $f$ also shares the same optimality guarantees: The same telescopic argument \eqref{eq:tech2} works with $f$ in place of $g$ and implies asymptotic feasibility, where the nonnegative property of $g$ is now replaced by the global lower bound $f\geq f_{\text{low}}$ due to (F1) (also see \cite[Lemma 1]{FletcherLeyfferToint2002}).
\end{remark}

\subsection{The filter as a feasibility mechanism}
\label{subsec:feasibility}

A significant characteristic of the developed algorithm is the separation between the filter and funnel mechanisms. We have already showed that the filter guarantees asymptotic feasibility, whereas the funnel function enforces optimality in the limit. Under minor modifications in the introduced algorithm, the funnel mechanism can be removed completely, without affecting the filter's primary role. In other words, our algorithm can be simplified with the aim of determining feasibility for \eqref{eq:milp}. To be more specific, during the first step of each outer iteration $k\geq 0$, the algorithm now runs all open nodes through the filter, instead of just the ones whose points satisfy the funnel criterion. If no filter-admissible point is found, the algorithm moves to the restoration phase. Suppose that the
funnel constraint is removed from the admissibility set $\Acal_k^{\res}(N)$. Besides looking for a filter-admissible point, the restoration oracle should now search for mixed-integer feasible points instead of node-optimal ones. Thus, the restoration outcomes \hyperlink{or1}{(OR1)} and \hyperlink{or1}{(OR3)} remain unchanged, while outcome \hyperlink{or1}{(OR2)} is now replaced by the much weaker one of returning a feasible solution, which immediately terminates the main algorithm. The restoration mechanism is then closer in spirit to those seen in the literature of filter methods \cite{fletcher2002nonlinear}. Lemmas \ref{lem:historical}-\ref{lem:coverage} are still valid (the latter two of which under minor modifications to account for the different second outcome of the restoration phase; note that $X_{<U_k}$ becomes $X$), and the telescopic argument of Theorem \ref{th:main} still guarantees that every infinite sequence $\{x_k\}$ satisfies $h(x_k)\to0$.  Every accumulation point of the sequence is thus a mixed-integer feasible solution for \eqref{eq:milp}. In fact, if $h(x_k)\to 0$, feasibility can be retrieved in finitely many (outer) iterations by following the feasibility recovery procedure described in Section \ref{sec:complexity}.

\section{Iteration complexity}\label{sec:complexity}

In this section, we provide an explicit control bound on the number of admissions across permitted executions, under a mild assumption and with a simple recovery procedure invoked within every outer step. We first need a standard notion of approximate optimality (e.g., see \cite{testa2019distributed}):
\begin{definition}\label{def:optimal}
    Suppose $X^*\neq \varnothing$ and let $\varepsilon\in (0,1)$. A feasible solution $x\in X$ of \eqref{eq:milp} is called \textit{$\varepsilon$-optimal} if it satisfies $f(x)-z^*\leq \varepsilon$.
\end{definition}
In order for the main algorithm to return an approximate optimal solution, it needs to recover a mixed-integer feasible point whenever the algorithm converges towards optimality along a fractional trajectory. This can be done by calling Algorithm \ref{alg:recovery} (described below) at the end of every outer iteration, right after construction of node $Z$ and before the final update of the algorithmic quantities. At a conceptual level, the introduced recovery procedure rounds the integer coordinates to a nearest integer vector (by component-wise nearest rounding) and optimizes the continuous coordinates by solving the corresponding fixed-integer LP over the original constraints, in similar vein to the incumbent update process \eqref{eq:update1}-\eqref{eq:update2}. If this LP is feasible, its solution updates the incumbent, and the bound $U_k-\ell_k$ is tested against the prescribed optimality tolerance. We ought to mention that rounding followed by a fixed-integer LP solve is an established ingredient of many mixed-integer optimization heuristics  \cite{fischetti2010heuristics}.

\begin{algorithm}[!htbp]
\caption{Optimality Recovery}\label{alg:recovery}
\small
\begin{algorithmic}[]
\Require Admitted point $x_{k+1}=(z_{k+1},y_{k+1})$, incumbent $U_k$, lower bound $\ell_k$, optimality tolerance $\eps>0$.
\State Choose $m_{k+1}\in\argmin_{m\in\Z^p}\norm{z_{k+1}-m}_\infty$.
\State Solve the original fixed-integer LP \eqref{eq:update1} at $z=m_{k+1}$.
\If{$Y(m_{k+1})=\varnothing$}
  \State \Return to Algorithm \ref{alg:bf}.
\EndIf
\State Let $\widehat y_{k+1}$ be the optimizer of \eqref{eq:update1} and set $\widehat x_{k+1}:=(m_{k+1},\widehat y_{k+1})$.
\State Set $U_k\gets\min\{U_k,f(\widehat x_{k+1})\}$ and store the corresponding incumbent point $x_{k+1}^*\in X$.
\If{$U_k\le\ell_k$}
  \State \Return $U_k=z^*$ and \textbf{terminate} with $x^*_{k+1}\in X^*$.
\ElsIf{$U_k-\ell_k\le\eps$}
  \State \textbf{terminate} with the $\eps$-optimal solution $x_{k+1}^*\in X$.
\Else
\State \Return to Algorithm 1 with the updated incumbent $U_k$.
\EndIf
\end{algorithmic}
\end{algorithm}

The main goal of this section is to establish that (BF) along with Algorithm \ref{alg:recovery} will return an $\varepsilon$-optimal solution in at most $\mathcal{O}(\varepsilon^{-1})$ outer iterations. This complexity bound must depend on universal constants that are problem-dependent but not execution-dependent. To this end, we assume that $g$ is uniformly bounded over the space of successful iterations:
\vspace{-0.1cm}
\begin{description}\hypertarget{f4}{}
\item[\hyperlink{f4}{(F4)}] There exists $\kappa_g<+\infty$ such that $g_k\leq \kappa_g$ for every $k\geq 0$.
\end{description}
\vspace{-0.1cm}
This is a mild assumption that is commonly satisfied under various settings in the relevant literature: Clearly, if the original MILP problem is infeasible, then one may simply set $\kappa_g:=\delta$. If $X\neq \varnothing$ and a feasible solution $x^{\text{feas}}$ is known apriori, then (BF) can initialize $U_0:=f(x^{\text{feas}})$. In that case, we have $g_k= U_k-f_k\leq f(x^{\text{feas}})-f_{\text{low}}=:\kappa_g$, by monotonicity. \hyperlink{f4}{(F4)} is also satisfied if the set $\Mcal:=\{m\in\Z^p:Y(m)\ne\varnothing\}$ is nonempty and finite, which is always the case when $X\neq\varnothing$ and the integer variables of \eqref{eq:milp} are bounded over $P$ (see \hyperlink{f5}{(F5)}). Indeed, the value $U^*:=\max_{m\in\Mcal}v(m)$ is well-defined and finite, and the constant $\kappa_g:=\max\{\delta,U^*-f_{\text{low}}\}$ controls $g$ over the space of all filter-admissible iterates.

We also assume a linear funnel function $\omega:[0,1/2]\to \mathbb{R}_{+}$ of the form
\begin{equation}\label{eq:linear-funnel}
 \omega(h):=\kappa_\omega h,\quad\text{with}\;\; \kappa_\omega\geq 2.
\end{equation}
In general, a larger \textit{funnel radius} $\kappa_{\omega}$ relaxes the funnel mechanism and increases its ``node span" during the search process. As the integer constraint violation decreases to zero, the funnel becomes tighter, and the algorithm becomes more selective. As we will see in Section \ref{sec:experiments}, a larger radius usually allows the restoration phase to identify a filter-admissible point faster, but a smaller $\kappa_\omega$ offers better control and faster convergence to optimality (also see Theorem \ref{th:complexity} below). We resort here to the lower bound $\kappa_\omega\geq 2$ for normalization purposes.

The main iteration complexity result of this section follows. In order to be fully rigorous and avoid the possibility of landing near an approximate optimal solution during the initial phase, we explicitly assume that the feasibility recovery mechanism is also invoked at the end of initialization, before the first outer iteration, with $x_0$, incumbent $U_0$ and reference $f_{\text{low}}$.

\begin{theorem}\label{th:complexity}
    Suppose that \hyperlink{f1}{(F1)}-\hyperlink{f4}{(F4)} hold and let $\varepsilon\in (0,1)$. In addition, suppose that $X\neq \varnothing$, the funnel function satisfies \eqref{eq:linear-funnel} and that Algorithm \ref{alg:recovery} is called right after replacement construction of node $Z$, as well as at the end of the initial phase. Then, Algorithm \ref{alg:bf} (BF) returns an $\varepsilon$-optimal feasible solution, or terminates earlier with a global optimum, in at most
    \begin{equation*}
        N_{\varepsilon}:=\ceil[\Bigg]{\frac{\log\left(1/\chi(\varepsilon)\right)}{\log\left(1/\beta\right)}+\frac{2\kappa_g(2-\beta)}{\gamma(1-\beta)\chi(\varepsilon)}}+1
    \end{equation*}
total outer iterations, where $ \chi(\eps):=\min\left\{\mu_Q,\displaystyle\frac{2\eps}{\kappa_{\omega}+C^*}\right\}$, for $\mu_Q$ and $C^*$ as in \eqref{eq:measure} and \eqref{eq:constant}, respectively.
\end{theorem}

\begin{proof}
Fix $\chi\in(0,1/2)$. For $l\geq 0$, define $u_l:=\beta^l/2$, $\alpha_l:=\max\{\chi,\beta^{l+1}/2\}$, and $I_l:=(\alpha_l,u_l]$. For $R_\chi:=\ceil[\Big]{\log\left(1/(2\chi)\right)/\log\left(1/\beta\right)}$, 
the disjoint strips $I_0,I_1,...,I_{R_\chi-1}$ cover the interval $(\chi,1/2]$. Observe that $R_\chi\geq 1$, where $R_\chi=1$ corresponds to the case $\beta/2\leq \chi<1/2$. Let $k_1<...<k_s$ be all iterations in the strip $I_l$ for some fixed $l\in \{0,...,R_\chi-1\}$. For $i<j$, we have $\beta h_{k_i}\le\beta u_l\le \alpha_l<h_{k_j}$, by definition. Hence,
\begin{align}\label{eq:sharp-telescope}
 (s-1)\gamma \alpha_l
 &\le\gamma\sum_{j=1}^{s-1}h_{k_{j+1}} \notag \\
 &\le\sum_{j=1}^{s-1}(g_{k_j}-g_{k_{j+1}})\notag\\
 &=g_{k_1}-g_{k_s} \notag \\
 &\le \kappa_g,
\end{align}
where we invoked Lemma \ref{lem:historical} in the second inequality, and non-negativity of $g$ and assumption \hyperlink{f4}{(F4)} in the last inequality. Therefore, this particular strip $I_l$ contains at most $1+\kappa_g/(\gamma a_l)$ filter-admissible iterates.

Next, we define the set  $\mathcal{S}_{\chi}:=\{k\geq 0:h_k>\chi\}$. By repeating the above argument for every interval of the form $I_l$ that covers $(\chi,1/2]$, we deduce that $\mathcal{S}_\chi$ is finite. In particular, we have
\begin{equation}\label{eq:strip-sum}
    |\mathcal{S}_\chi|\leq R_\chi+\frac{\kappa_g}{\gamma}\sum_{l=0}^{R_\chi-1}\frac1{\alpha_l}.
\end{equation}
All strips except the last have $\alpha_l=\beta^{l+1}/2$, and the last one has $\alpha_{R_\chi-1}=\chi$. Consequently, with an empty sum interpreted as zero when $R_\chi=1$, we obtain
\begin{align}\label{eq:geometric-sum}
 \sum_{l=0}^{R_\chi-1}\frac1{\alpha_l}
 &=\frac1\chi+2\sum_{l=1}^{R_\chi-1}\beta^{-l}\notag\\
 &=\frac1\chi+\frac{2(\beta^{-(R_\chi-1)}-1)}{1-\beta}\notag\\
 &\le\frac1\chi+\frac1{(1-\beta)\chi} \notag \\
 &=\frac{2-\beta}{(1-\beta)\chi},
\end{align}
where the inequality used the relation $\beta^{R_\chi-1}>2\chi$, the latter of which follows from $\chi<1/2$, $\beta<1$, and the definition of $R_\chi$. From \eqref{eq:strip-sum} and \eqref{eq:geometric-sum} we deduce that
\begin{equation}\label{eq:chi}
    |\mathcal{S}_\chi|\leq \ceil[\Bigg]{\frac{\log\left(1/(2\chi)\right)}{\log\left(1/\beta\right)}+\frac{\kappa_g(2-\beta)}{\gamma(1-\beta)\chi}}
\end{equation}
for every $\chi\in (0,1/2)$. Given $\varepsilon\in (0,1)$, let $\chi(\varepsilon)$ be defined as in the statement. Recall that $\varepsilon\in (0,1)$, $\kappa_{\omega}\geq 2$ by \eqref{eq:linear-funnel}, and $C^*\geq 0$ from \eqref{eq:constant}, so $\chi(\varepsilon)<1$. Therefore, we can apply \eqref{eq:chi} for $\chi:=\chi(\varepsilon)/2$ and deduce that (BF) will output a point $x_{k+1}=(z_{k+1},y_{k+1})$ such that $h_{k+1}\leq \chi(\varepsilon)/2$ in at most $N_\varepsilon$ outer iterations.

Now, observe that, for that point $x_{k+1}$, we must have $h_{k+1}\leq \mu_Q/2$, so $m_{k+1}\in Q$ and $Y(m_{k+1})\neq \varnothing$ by \eqref{eq:round-contradiction}, where $m_{k+1}$ represents the closest integer vector to $z_{k+1}$. Algorithm \ref{alg:recovery} is called after $x_{k+1}$ is admitted and node $Z$ is constructed by \eqref{eq:replacement}, and returns a point $\widehat{x}_{k+1}=(m_{k+1},\widehat{y}_{k+1})\in X$. In similar fashion to \eqref{eq:tech4}, the computed feasible point satisfies
\begin{equation}\label{eq:almost}
   f(\widehat{x}_{k+1})\leq f(x_{k+1})+C^* h(x_{k+1}).
\end{equation}
Since $X\neq \varnothing$ and (BF) does not terminate at iterate $k$ with an exact global solution, Lemma \ref{lem:ellbound} yields $\ell_k\leq z^*$. The funnel property (see \eqref{eq:funnelproperty}), \eqref{eq:almost} and bound $h_{k+1}\leq \chi(\varepsilon)/2$ yield
\begin{align*}
    f(\widehat{x}_{k+1})-z^*&\leq f(\widehat{x}_{k+1})-\ell_k \\
    &\leq f(x_{k+1})-\ell_k +C^* h(x_{k+1}) \\
    &\leq (\kappa_{\omega}+C^*) h(x_{k+1}) \\
    &\leq (\kappa_{\omega}+C^*) \chi(\varepsilon)/2 \\
    &\leq \varepsilon.
\end{align*}
Therefore, $\widehat{x}_{k+1}$ is an $\varepsilon$-optimal feasible solution, and the complexity result follows.
\end{proof}

When the optimality threshold $\varepsilon$ is sufficiently small, we get $\chi(\eps)=2\eps/(\kappa_{\omega}+C_*)$. Observe that any $\varepsilon>0$ strictly smaller than $1/\max_{i} \norm{B_i}_1$ satisfies the latter due to \eqref{eq:rho}. The final iteration complexity bound then becomes
\begin{equation*}
    N_{\eps}= \mathcal{O}\left(\frac{\log\left(1/\varepsilon\right)}{\log\left(1/\beta\right)}+\frac{\kappa_g(2-\beta)(\kappa_{\omega}+C^*)}{\gamma(1-\beta)\varepsilon}\right).
\end{equation*}
Thus, by ignoring logarithmic terms, the worst-case number of iterations until the algorithm outputs an $\varepsilon$-optimal solution (or an exact optimal if it terminates earlier) is of order $\mathcal{O}(\varepsilon^{-1})$. 

We should lastly mention that, similarly to asymptotic optimality (Remark \ref{rem:f1}), the objective function $f$ can replace the metric $g$, in which case (BF) admits a tighter iteration complexity bound of order $\mathcal{O}\left(\log(\varepsilon^{-1})\right)$. This is due to the fact that monotonicity of the sequence of lower bounds $\ell_k$ and the funnel property directly yield $f_i-f_j\leq \kappa_\omega h_i$ for every $i<j$. The same telescoping argument \eqref{eq:sharp-telescope} yields $(s-1)\gamma \alpha_l\leq f_{k_l}-f_{k_s}\leq \kappa_\omega u_l$, and thus $|\mathcal{S}_\chi|\leq (1+\frac{\kappa_\omega}{\beta\gamma})R_\chi$.

\section{A realization of the restoration oracle}\label{sec:restoration}

The restoration oracle does not prescribe a particular solver. Under explicit finite bounds on the integer variables, an exact (BB) solver on the selected node is an obvious candidate for realizing the oracle, as it either proves the absence of an incumbent-improving point or establishes a node optimum in finitely many steps \cite{belotti2013mixed}. The restoration phase's novel goal, however, is reflected upon the fact that the procedure can terminate at a much earlier iteration with the detection of a fractional filter-admissible point, thus avoiding an exhaustive search for an optimal mixed-integer solution or a node infeasibility certificate. As a result, during implementation, an iterative (BB) or (BC) algorithm could perform the filter and funnel tests at every generated ``suboptimal" solution. 

The purpose of this section is to introduce a realization of the restoration oracle that primarily focuses on identifying filter-funnel-admissible points. At a high-level, our proposed algorithm starts from the selected node’s relaxation and progressively removes regions containing unsuccessful candidates. At each inner iteration, a solver minimizes the objective function over the current search region. If the solution is fractional and fails the filter or funnel test, then the method excludes a number of cells surrounding it, while preserving within those cells all mixed-integer feasible points and all points satisfying both acceptance requirements. This process continues until the method finds an admissible fractional point, obtains a node-optimal mixed-integer solution, or certifies that the node contains no feasible solution improving the incumbent.

In order to guarantee that the proposed algorithm will always output one of \hyperlink{or1}{(OR1)}-\hyperlink{or1}{(OR3)} in finitely many iterations (i.e., that \hyperlink{f3}{(F3)} holds), we need to impose additional conditions on the structure of the original problem. We therefore make the following assumption:
\vspace{-0.1cm}
\begin{description}\hypertarget{f5}{}
\item[\hyperlink{f5}{(F5)}] There exist $\underline{z},\overline{z}\in \mathbb{Z}^{p}$ such that $\underline{z}\leq z \leq \overline{z}$ for every $(z,y)\in P$.
\end{description}
\vspace{-0.1cm}
This assumption is a bounded-integer-variable restriction, with direct precedents in various works in the relevant literature (for instance, see \citep{chen2011finite, dash2018binary}). It is important to note that this condition is imposed on the relaxations, not only on mixed-integer feasible points. The distinction matters when the latter set is empty. Continuous variables may remain unbounded, which makes \hyperlink{f5}{(F5)} weaker than the more common assumption that $P$ is a compact polyhedron \cite{sun2024decomposition, testa2019distributed}.

\subsection{The cell-deletion algorithm}\label{sec:znode}

Fix a restoration instance at iteration $k$ and let $N_k\in \mathcal{O}_k$ be the selected node. In the formulas below, we write $N=N_k$ for simplicity of exposition. The quantities $U_k$, $\ell_k$, and $\F_k$ are also frozen for this call, and they are recomputed in Algorithm \ref{alg:rp} if a subsequent invocation processes a different surviving open node. We also assume a linear funnel given by \eqref{eq:linear-funnel}.

Before entering the full cell-deletion procedure mentioned above (and explicitly described below), Algorithm \ref{alg:fecc} first performs an early inexpensive test that requires solving $2p$ LPs. More specifically, let $x_N=(z_N,y_N)$ be the stored LP minimizer of $N$, and let $m_N\in \mathbb{Z}^{p}$ be the nearest integer vector to $z_N$. Also let $V_p=\{\pm e_1,...,\pm e_p\}$, where $e_i\in \mathbb{R}^{p}$ is the unit vector whose $i$-th coordinate is equal to 1, and let $\alpha_k:=\beta \min_{i\in \mathcal{I}_k} h_i$, where $\mathcal{I}_k$ denotes the set of all active filter indices. At its initial stage, the introduced algorithm solves the LP ``$\min f(x)$ subject to $x=(z,y)\in W_{k,N}(v)$" for every unit direction $v\in V_p$, where
\begin{equation}\label{eq:www}
    W_{k,N}(v):=\left\{ x\in P(N): -\alpha_k\mathbf{1}\leq z-m_N\leq \alpha_k \mathbf{1},\; f(x)-\ell_k\leq \kappa_w v^\top (z-m_N),\; f(x)\leq U_k \right\}.
\end{equation}
If any of these $2p$ LPs is feasible, then the algorithm outputs either a global optimal solution or outcome \hyperlink{or3}{(OR3)}, depending on the integrality violation of the solution. Otherwise, if all LPs are infeasible, we move to the proposed method's full iterative stage described right below.

For $m\in\Z^p$, we define the closed nearest-integer cell $C_m:=m+[-\tfrac12,\tfrac12]^p$, and its node restriction by
\begin{equation}\label{eq:cells}
 P_N[m]:=\{(z,y)\in P(N):z\in C_m\}.
\end{equation}
Further, let 
\begin{equation}
    \Lcal_N:=\{m\in\Z^p:P_N[m]\ne\varnothing\}
\end{equation}
represent the family of labels of all closed cells that intersect the node relaxation. By \hyperlink{f5}{(F5)}, it is a subset of the integer box $[\underline z,\overline z]\cap\Z^p$, so it is finite. Every real point belongs to at least one closed nearest-integer cell, so $P(N)=\bigcup_{m\in\Lcal_N}P_N[m]$. At a tie, a point may belong to several cells; this overlap is retained in the construction. Lastly, we define the \textit{protected set} within node $N$, which includes all filter-funnel-admissible points as well as all mixed-integer feasible solutions, 
\begin{equation}\label{eq:rset}
    \mathcal{R}_{k,N}:=\Acal_k^{\res}(N)\cup X(N).
\end{equation}
At every inner iteration $r\geq 0$, the proposed algorithm minimizes the objective function over the \textit{search region}
\begin{equation}\label{eq:protected}
 \mathcal{D}_{k,N,r}:=\mathcal{R}_{k,N}\cup
       \bigcup_{m\in\Lcal_N\setminus S_{k,r}}P_N[m].
\end{equation}
Here, $S_{k,r}$ represents the set of all excluded cell labels before iteration $r$, that is, of all ``neighboring" cells related to all previously rejected points. To be more precise, let $x_{k,r}$ be the solution of the problem $\min_{x\in \mathcal{D}_{k,N,r}} f(x)$ at inner iteration $r$. If that point is rejected, meaning that it is not filter-funnel-admissible or node-optimal, we exclude all still-present cell labels that touch it, while preserving protected points. Specifically, define
\begin{equation}\label{eq:touching}
 B_{k,r}(x_{k,r}):=\{m\in\Lcal_N\setminus S_{k,r}:x_{k,r}\in P_N[m]\}
\end{equation}
and pick a set $T_{k,r}$ such that $B_{k,r}(x_{k,r})\subseteq T_{k,r}\subseteq\Lcal_N\setminus S_{k,r}$. Then, we set
\begin{equation}\label{eq:update}
    S_{k,r+1}:=S_{k,r}\cup T_{k,r}
\end{equation}
and repeat the process. Excluding all touching labels ensures that the failed point $x_{k,r}$ belongs to no remaining cell afterward, and shared cell boundaries do not permit the same point to persist through a different touching cell. An illustration of this process is given in Figure \ref{fig:fecc} below.

Since every term in \eqref{eq:protected} is contained in $P(N)$, one has
\begin{equation}\label{eq:protected-inclusion}
 X(N)\subseteq\mathcal{R}_{k,N}\subseteq\mathcal{D}_{k,N,r}\subseteq P(N).
\end{equation}
In fact, $\mathcal{D}_{k,N,r+1}\subseteq\mathcal{D}_{k,N,r}$, so the corresponding optimal values are nondecreasing, with value $+\infty$ assigned to an empty set. With the initialization $S_{k,0}=\varnothing$, we have $\mathcal{D}_{k,N,0}=P(N)$. The above process along with the inclusion $X(N)\subseteq\mathcal{D}_{k,N,r}$, whenever $X(N)\neq\varnothing$, yield
\begin{equation}\label{eq:FECC-monotone}
f(x_{N})\leq f(x_{k,1})\leq f(x_{k,2})\leq \cdots \leq f(x_{k,r})\leq \cdots
\le v_N:=\min_{x\in X(N)}f(x).
\end{equation}
In other words, at each inner iteration, the cell-deletion algorithm is getting monotonically closer to a node-optimal solution, while searching for filter-funnel admissible points at the same time. The full algorithm is described below.  

\begin{figure}[H]
    \centering
    \includegraphics[width=\textwidth]{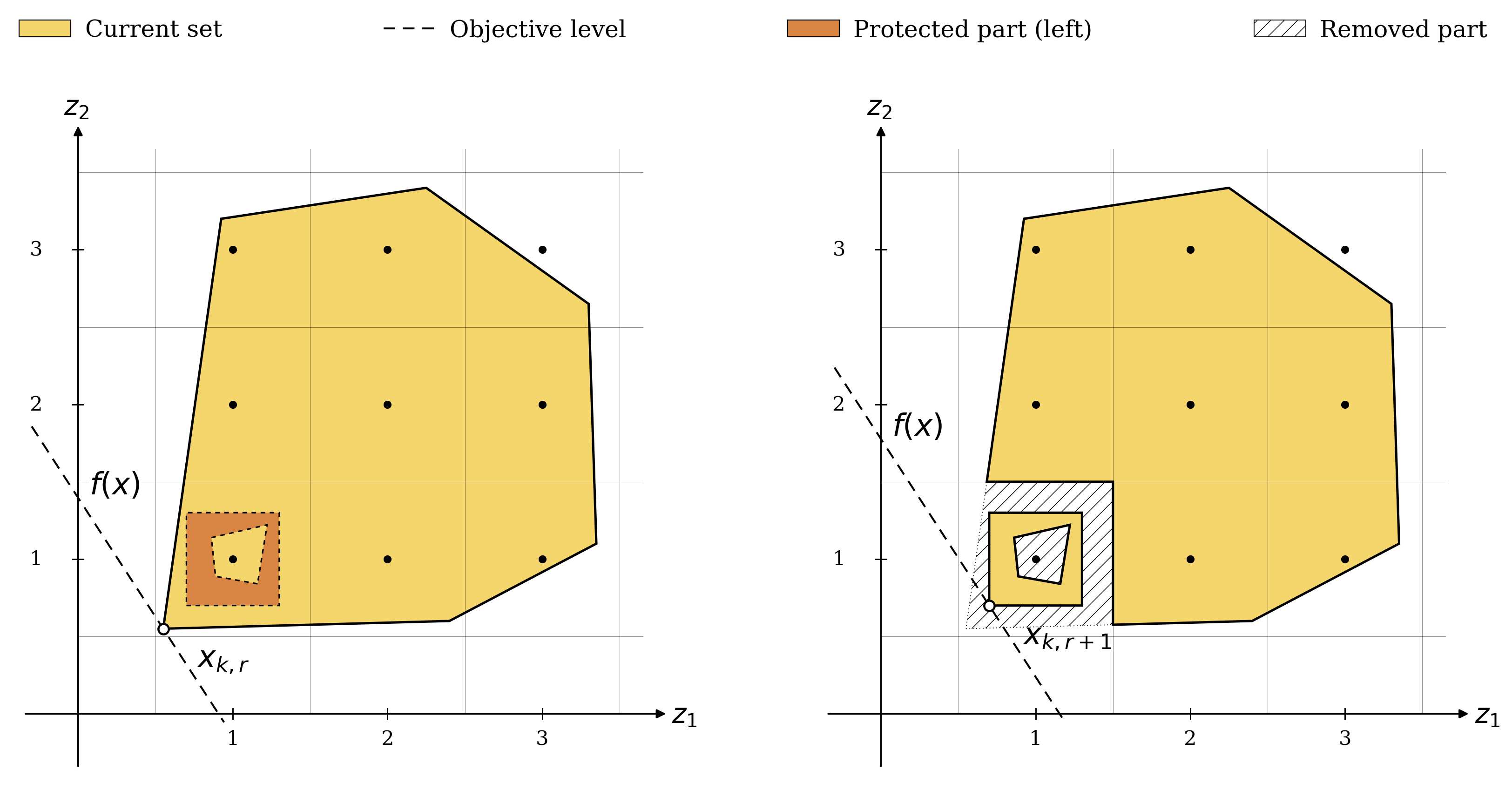}
    \caption{\small One cell-deletion iteration for a MILP problem with dimensions $p=2,\;q=0$. Within the selected cell centered at $(1,1)$, the conditions imposed by $h$ define a square in this particular instance. Intersecting this square with the funnel set produces the orange protected region, whose central ``hole" contains inadmissible fractional points. The point $(1,1)$ remains protected separately. Excluding the cell removes only its unprotected portion (hatched), leaving the schematic yellow region shown on the right.}
    \label{fig:fecc}
\end{figure}

\begin{algorithm}[H]
\caption{Cell-Deletion (CD)}\label{alg:fecc}
\small
\begin{algorithmic}[]
\Require Node $N\in \mathcal{O}_k$, filter $\F_k$, incumbent $U_k$, funnel radius $\kappa_\omega\geq 2$.
\State For every $v\in V_p$ solve $\min_{x\in W_{k,N}(v)} f(x)$, where $W_{k,N}(v)$ is given by \eqref{eq:www}.
\If {$x_k^v\in W_{k,N}(v)$ is an optimal solution for some $v\in V_p$}
\If{$h(x_k^v)>0$}
\State \Return \hyperlink{or3}{(OR3)} with $x_{k,r}$ to Algorithm \ref{alg:rp}.
\Else
\State \textbf{terminate} all algorithms with $x_k^v\in X^*$ and $f(x_k^v)=z^*$.
\EndIf
\EndIf
\State Construct the finite cell family and protected set in \eqref{eq:cells}--\eqref{eq:rset}; set $S_{k,0}:=\varnothing$.
\For{$r=0,1,2,\ldots$}
  \State Solve the inner subproblem
  \begin{equation}\label{eq:sub}
      \varphi_{k,N,r}:=\min_{x\in \mathcal{D}_{k,N,r}} f(x)
  \end{equation}
  \If{$\varphi_{k,N,r}=+\infty$}
    \State \Return \hyperlink{or1}{(OR1)} to Algorithm \ref{alg:rp}.
  \EndIf
  \State Select an optimizer $x_{k,r}$ of \eqref{eq:sub}.
  \If{$U_k<+\infty$ and $\varphi_{k,N,r}\ge U_k$}
    \State \Return \hyperlink{or1}{(OR1)} to Algorithm \ref{alg:rp}.
  \ElsIf{$h(x_{k,r})=0$}
    \State \Return \hyperlink{or2}{(OR2)} with $x_{k,r}$ and its node-optimality certificate to Algorithm \ref{alg:rp}.
  \ElsIf{$x_{k,r}\in\Acal_k^{\res}(N)$}
    \State \Return \hyperlink{or3}{(OR3)} with $x_{k,r}$ and the certificate $\varphi_{k,N,r}\le f(w)$ for all $w\in X(N)$ to Algorithm \ref{alg:rp}.
  \Else
    \State Compute $B_{k,r}(x_{k,r})$ and choose $T_{k,r}$ such that $B_{k,r}(x_{k,r})\subseteq T_{k,r}\subseteq\Lcal_N\setminus S_{k,r}$.
    \State Set $S_{k,r+1}:=S_{k,r}\cup T_{k,r}$.
  \EndIf
\EndFor
\end{algorithmic}
\end{algorithm}

It is crucial to verify that the algorithm is well-defined, in the sense that all tests conducted within every inner iteration are valid: If the main feasible region $\mathcal{D}_{k,N,r}$ is empty, then the inclusion $X(N)\subseteq\mathcal{D}_{k,N,r}$ \eqref{eq:protected-inclusion} directly gives $X(N)=\varnothing$, so \hyperlink{or1}{(OR1)} is valid for either incumbent regime. If the subproblem \eqref{eq:sub} admits a solution $x_{k,r}$, then the latter belongs to $\mathcal{D}_{k,N,r}\subseteq P(N)$, again by \eqref{eq:protected-inclusion}. A finite incumbent with $f(x_{k,r})\ge U_k$ and \eqref{eq:FECC-monotone} give $X_{<U_k}(N)=\varnothing$, validating \hyperlink{or1}{(OR1)}. The initial $U_k\le\ell_k$ test is valid for the same reason, using $f\ge\ell_k$ on $P(N)$. If this finite-incumbent test fails and $h(x_{k,r})=0$, then $x_{k,r}\in X(N)$, so \eqref{eq:FECC-monotone} proves its node optimality. Note that its value $f(x_{k,r})$ is strictly below $U_k$, as required for \hyperlink{or1}{(OR2)}. If the integrality test fails but $x_{k,r}\in\Acal_k^{\res}(N)$, it is a fractional admissible point, so \hyperlink{or1}{(OR3)} is valid.

If all the above tests fail, then the algorithm moves to the next inner iteration. Under \hyperlink{f5}{(F5)}, we now show that the cell-deletion algorithm will eventually terminate with exactly one of the outcomes required by the restoration oracle, in at most finitely many steps.

\begin{proposition}\label{prop:FECC}
\textit{Suppose that \hyperlink{f1}{(F1)}, \hyperlink{f2}{(F2)} and \hyperlink{f5}{(F5)} hold. Also, assume that the subproblem \eqref{eq:sub} is solved exactly in every inner iteration. Then, Algorithm \ref{alg:fecc} terminates with exactly one of \hyperlink{or1}{(OR1)}-\hyperlink{or1}{(OR3)}, in finitely many iterations.}
\end{proposition}

\begin{proof}
Let $k\geq 0$ be a fixed outer iteration and $N\in \mathcal{O}_k$ be the chosen node. First, suppose the algorithm returns some  $x_k^v=(z_k^v,y_k^v)\in \argmin_{x\in W_{k,N}(v)} f(x)$ for some $v\in V_p$ during its initial stage. Observe that $m_N$ is the closest integer vector to $z_k^v$, i.e., $h(x_k^v)=\norm{z_k^v-m_N}_{\infty}$, by construction of $W_{k,N}(v)$. We consider two cases:
\begin{itemize}
    \item[--] If $x_k^v\not\in X(N)$, then $h(x_k^v)=\norm{z_k^v-m_N}\leq \alpha_k \leq \beta h_i$ for every $i\in \mathcal{I}_k$, meaning that $x_k^v$ is filter-admissible. In addition, we have
    \begin{equation}\label{eq:sand}
        f(x_k^v)-\ell_k\leq \kappa_\omega v^\top (z_k^v-m_N)\leq \kappa_\omega \norm{z_k^v-m_N}=\omega(h(x_k^v)),
    \end{equation}
namely, $x_k^v\in \mathcal{G}_k$. Hence, the algorithm returns \hyperlink{or3}{(OR3)} along with the fractional point $x_k^v$.
\item[--] If $x_k^v\in X(N)$, then \eqref{eq:sand} yields $f(x_k^v)\leq \ell_k\leq f(x_N)\leq f(x_k^v)$, meaning that $x_k^v$ is a node-optimal solution. In fact, since $f(x_k^v)=\ell_k$ and $x_k^v\in X$, we have $f(x_k^v)=z^*$ due to Lemma \ref{lem:ellbound}, so (CD) terminates and the main algorithm returns $x_k^v\in X^*$.
\end{itemize}
If none of the $2p$ LPs has an optimal solution, then the algorithm moves to its full iterative stage. Suppose that, at inner iteration $r\geq 0$, the returned solution $x_{k,r}$ is not mixed-integer or filter-admissible, i.e., $x_{k,r}\notin X(N)$ and $x_{k,r}\notin\Acal_k^{\res}(N)$, meaning that neither \hyperlink{or1}{(OR2)} nor \hyperlink{or1}{(OR3)} are returned during iteration $r$. As $x_{k,r}\in \mathcal{D}_{k,N,r}\setminus \mathcal{R}_{k,N}$, the returned point must belong to at least one remaining cell $P_N[m]$. Thus, $B_{k,r}(x_{k,r})\ne\varnothing$, and any $T_{k,r}$ satisfying \eqref{eq:update} is a nonempty subset of the previously unexcluded labels. In other words, we have
\begin{equation}\label{eq:strict}
 S_{k,r}\subsetneq S_{k,r+1}\subseteq\Lcal_N,
 \quad \text{with}\quad |S_{k,r+1}|\geq |S_{k,r}|+1.
\end{equation}
Let $r^*\geq 0$ be the first inner iterate for which we have $\mathcal{D}_{k,N,r^*}=\mathcal{R}_{k,N}$. Note that we must have $r^*\leq |\Lcal_N|<+\infty$ (due to \hyperlink{f5}{(F5)}), since at every inner repetition at least one touching cell is removed by the cell-deletion algorithm \eqref{eq:strict}. If $\mathcal{D}_{k,N,r^*}$ is empty, \hyperlink{or1}{(OR1)} is returned. Otherwise, the next selected optimizer of \eqref{eq:sub} must belong to either $X(N)$ or $\Acal_k^{\res}(N)$. Therefore, unless \hyperlink{or1}{(OR1)} is returned at an earlier stage of the inner iterate $r^*$, the algorithm returns either \hyperlink{or1}{(OR2)} or \hyperlink{or1}{(OR3)}, and terminates with output vector $x_{k,r^*}$.
\end{proof}

Lastly, Algorithm \ref{alg:fecc} allows for a simple realization of \eqref{eq:replacement} when \hyperlink{or3}{(OR3)} is retrieved by the full iterative procedure. Suppose that $x_{k+1}:=x_{k,r}$ is the returned point with the certificate $f(x)\geq f(x_{k+1})$ over $X(N_k)$. By the last failed \hyperlink{or1}{(OR1)} test, this point improves on the incumbent when the latter is finite. The artificial replacement node $Z$ can then be defined with $X(Z)$ as in \eqref{eq:replacement}, and a relaxed polyhedral region given by $P(Z):= \left\{x\in P(N_k):f(x_{k+1})\le f(x)\le U_k\right\}$.

\subsection{An exact polyhedral representation of the inner subproblem}

The success of the proposed algorithm in fully realizing assumption \hyperlink{f3}{(F3)} lies in solving the subproblem \eqref{eq:sub} exactly at every inner iteration. Although the objective is still linear, the feasible region $\mathcal{D}_{k,N,r}$ might not admit a closed convex polyhedral representation. In this subsection, we provide a union polyhedral representation of $\mathcal{D}_{k,N,r}$, when the funnel function $\omega$ is chosen as in \eqref{eq:linear-funnel}. This in turn allows one to solve \eqref{eq:sub} either by solving multiple (finitely many) LPs and picking the minimum value, or via classical disjunctive programming methods \cite{balas1979disjunctive}. 

First, note that the search region
admits the exact representation
\begin{equation}\label{eq:union}
\mathcal D_{k,N,r}
=\left(\bigcup_{m\in\mathcal L_N\setminus S_{k,r}}P_N[m]\right)\cup\left(\bigcup_{m\in S_{k,r}}\{(z,y)\in P(N):z=m\}\right)\cup\left(\bigcup_{m\in S_{k,r}}
\bigl(\mathcal A_k^{\mathrm{res}}(N)\cap P_N[m]\bigr)
\right).
\end{equation}
The first family of the above union retains the full relaxation within each surviving cell, whereas
the second retains the integer fiber of each excluded cell. Minimizing the objective function over each of the two families results in two LP families of the form
\begin{subequations}\label{eq:fecc-basic-lps}
\begin{align}
\min\;&f(z,y)&&\text{s.t. }\;\;m-\tfrac12\mathbf1\le z\le m+\tfrac12\mathbf1,\quad m\in\mathcal L_N\setminus S_{k,r}, \quad (z,y)\in P(N);\\
\min\;&f(z,y)&&\text{s.t. }\;\;z=m,\quad m\in S_{k,r}, \quad (z,y)\in P(N).
\end{align}
\end{subequations}
The third family of \eqref{eq:union} retains precisely the admissible portion of each excluded
cell. In order to show that this can also be written as a polyhedral union, we first order the non-dominated active filter records by increasing first coordinate. Specifically, assuming that $s_k$ represents the number of pairs in $\mathcal{F}_k$, i.e., $s_k:=|\mathcal{I}_k|$, we write their coordinates in that order as $(h_{k,i}, g_{k,i})$, and we then set $a_{k,i}:=\beta h_{k,i}$ and $b_{k,i}:= g_{k,i}$, so that we have
\begin{equation*}\label{eq:stair-parameters}
 0<a_{k,1}<\cdots<a_{k,s_k},\qquad b_{k,1}>\cdots>b_{k,s_k}.
\end{equation*}
The strict orders follow after removing dominated pairs and duplicate ties. If two pairs have the same first coordinate, the one with smaller second coordinate dominates, while, if their second coordinates are equal, then the one with smaller first coordinate dominates. For convenience purposes, we also set $a_{k,s_k+1}=1/2$. The staircase structure of $\mathcal{F}_k$ (also seen in Figure \ref{fig:filter}) boils down to the constraints
$h(x)\le a_{k,j+1}$, together with
$g(x,U_k)+\gamma h(x)\le b_{k,j}$ when $j\ge1$, for
$j=0,...,s_k$. Note that the last alternative related to the integer-constraint violation
is simply the universal bound $h\le1/2$. Equivalence with the filter follows by
choosing the largest index whose $h$ test fails, or $j=0$ if none fails. Moreover, observe that the funnel part of $\mathcal{A}_k^{\text{res}}$ includes constraints of the form $f(x)-\ell_k\le\kappa_\omega h(x)$. 

Therefore, in order to retrieve an efficient representation of the last family of \eqref{eq:union}, it suffices to describe inequalities involving the integer violation $h$ in terms of linear constraints. For each fixed triple
$(m,v,j)\in S_{k,r}\times V_p\times\{0,...,s_k\}$, consider the problem
\begin{equation}\label{eq:fecc-admissible-lp}
\begin{aligned}
\min_{x,t}\quad & f(x)\\
\text{s.t.}\quad
&-t\mathbf1\le z-m\le t\mathbf1\\
&0\le t\le a_{k,j+1}\\
&f(x)-\ell_k\le\kappa_\omega v^\top(z-m)\\
&f(x)\le U_k\\
&g(x,U_k)+\gamma t\le b_{k,j}\quad(\text{if }j\ge1)\\
&x=(z,y)\in P(N).
\end{aligned}
\end{equation}
By \eqref{eq:g}, the last inequality constraint can be replaced by either one or two linear inequalities, depending on the value of the current incumbent. Thus each problem \eqref{eq:fecc-admissible-lp} (i.e., for each triple) is an ordinary LP with one extra continuous variable and
at most $2p+5$ extra scalar inequalities. Note that for every $x\in P_N[m]$ the following holds:
\begin{equation}\label{eq:max}
h(x)=\|z-m\|_\infty=\max_{v\in V_p}v^\top(z-m).
\end{equation}
A feasible point of \eqref{eq:fecc-admissible-lp} lies in $P_N[m]$ and
satisfies $h(x)\le t$ by the first constraint. The second and fifth constraints therefore imply the true filter
tests, whereas $v^\top(z-m)\le h(x)$ gives the true funnel inequality (third constraint) while $v$ varies in $V_p$, due to \eqref{eq:max}.
Conversely, for any admissible point $x$ in the union of the intersections $\mathcal{A}_k^{\text{res}}\cap P_N[m]$ over the set of excluded cells, select a satisfied
staircase alternative, set $t=h(x)$, and choose a maximizing $v$ in
\eqref{eq:max}. This yields a feasible point for
\eqref{eq:fecc-admissible-lp}. Therefore, the last part of \eqref{eq:union} can be written as the union of all feasibility sets of \eqref{eq:fecc-admissible-lp} taken over $S_{k,r}\times V_p\times \{0,...,s_k\}$.

Under \hyperlink{f1}{(F1)}, every nonempty component LP of the form \eqref{eq:fecc-basic-lps}-\eqref{eq:fecc-admissible-lp} has an attained finite minimum,
since its objective is bounded below on $P(N)$. Assigning value $+\infty$
to infeasible components and selecting an optimizer with the smallest value (among all optimal solutions of the above LPs) 
therefore solves the inner subproblem \eqref{eq:sub} exactly. Importantly, the selected point $x_{k,r}$ belongs to the original
union, and its value is a lower bound on $f$ over $X(N)$, preserving all
plausible outcome tests and certificates. Of course, if all components of the LP representatives are
infeasible, then the union is empty and the algorithm terminates with outcome \hyperlink{or1}{(OR1)}. 

Although each inner subproblem of Algorithm \ref{alg:fecc} can be solved exactly by this decomposition - either by solving finitely many LPs as above or via employing appropriate disjunctive programming techniques \cite{balas1979disjunctive}  - this particular procedure can be computationally expensive, as the next section also indicates. In particular, a direct evaluation at every inner iteration $r\geq 0$ admits a representation with at most
\begin{equation}\label{eq:bad-count}
|\mathcal L_N\setminus S_{k,r}|
+\bigl[1+2p(s_k+1)\bigr]|S_{k,r}|
\end{equation}
polyhedral components: one per surviving cell, and one fiber component plus $2p(s_k+1)$
admissible-point polyehdra per excluded label. In the worse-case scenario, the inner iterative restoration procedure may require an exhaustive exploration before the return of a valid outcome.

\section{Numerical experiments}\label{sec:experiments}

We now examine two aspects of (BF): its ability to obtain a valid restoration outcome, and its performance in solving MILPs to numerical optimality. The goal of the first set of experiments is to test the ability of Algorithm \ref{alg:fecc} to output a filter-admissible point faster than a state-of-the-art implementation of the oracle. In the second batch of numerical experiments, we evaluate the performance of three different variants of (BF) against SCIP and Gurobi, the latter of which are among the fastest open source and commercial MILP solvers, respectively. All experiments are performed on a personal laptop with a 2.10GHz 12-gen Intel Core i7-1260P processor and 16GB RAM, using Python 3.13.16 and Gurobi 13.0.3 for LP and auxiliary MILP solves \cite{gurobi2026}.

\subsection{Implementations of branch-and-filter}

In all conducted experiments, we test three different (BF) variants, which are explicitly characterized by their corresponding implementation of the restoration oracle: A branch-and-bound method (BF-BB), a branch-and-cut method (BF-BC), and the proposed cell-deletion algorithm of Section \ref{sec:restoration} (BF-CD). BF-BB and BF-BC construct the selected node’s mixed-integer subproblem and solve it with Gurobi node MILPs with maximum-infeasibility branching and solver cuts disabled/enabled,
respectively. Every exposed optimal node-relaxation callback is sequentially checked for a possible fractional
point that is filter-funnel-admissible. The BF-CD variant follows the structure of Algorithm \ref{alg:fecc}. Specifically, if the initial $2p$ LPs test fails, the implementation groups a rectangle of deleted labels and represents its protected union in a single indicator MILP. Gurobi then solves this disjunctive formulation of \eqref{eq:sub} using its native machinery, with cuts enabled. Under a \hyperlink{or3}{(OR3)} outcome, we construct $Z$ according to the discussion in Section \ref{sec:znode}. All three (BF) variants use dual simplex for LP solves, Gurobi presolves (Presolve$=$2), warm-starts (LPWarmStart$=$2) and primal heuristics (Heuristics$=$0.05) are enabled, and dual reductions are disabled. Gurobi feasibility and LP optimality tolerances are $10^{-9}$, its integer tolerance is $10^{-8}$, and requested relative and absolute MILP gaps are zero. Further, any independent feasibility or integrality checks use tolerance error $10^{-7}$, and filter tests use $10^{-8}$. Accordingly, optimality throughout this section is verified within a validated error gap of at most $10^{-6}$.

\subsection{Restoration experiments}

We replay the same first natural restoration state for 27 selected MIPLIB 3.0 problems \cite{bixby1998miplib}. We initially assume no known mixed-integer feasible solution, so the metric $g$ is defined according to the first case of \eqref{eq:g} with $\delta=0.01$. The entries ``$(2)$" and ``$(3)$" in the Outcome (Out.) columns of Table \ref{tab:1} refer to \hyperlink{or2}{(OR2)} and \hyperlink{or3}{(OR3)}. ``TL" denotes an unresolved time limit of $20$ seconds. Upon a returned node-optimal or filter-admissible point $x_N$, its corresponding reported objective value and integrality violation (within three decimal points) is given by $f_{\text{out}}$ and $h_{\text{out}}$, respectively. All tables presented below give complete charged wall times in seconds, with three decimal places, and bold text corresponds to the fastest resolved time per problem.

\begin{table}[H]
\centering
\begingroup
\fontsize{8.5}{9.8}\selectfont
\setlength{\tabcolsep}{3.0pt}
\renewcommand{\arraystretch}{1.02}
\begin{tabular}{|l|crrr|crrr|crrr|}
\hline
&\multicolumn{4}{c|}{BF-BB}&\multicolumn{4}{c|}{BF-BC}&\multicolumn{4}{c|}{BF-CD}\\
\cline{2-13}
Problem & Out. & Time (s.)\hspace{-0.3cm} & $f_{\rm out}$ & $h_{\rm out}$ & Out. & Time (s.)\hspace{-0.3cm} & $f_{\rm out}$ & $h_{\rm out}$ & Out. & Time (s.)\hspace{-0.25cm} & $f_{\rm out}$ & $h_{\rm out}$\\
\hline
p0201 & (3) & 0.049 & 7497.077 & 0.385 & (2) & 0.186 & 7615.000 & 0.000 & (3) & \textbf{0.007} & 7350.000 & 0.426\\
p0282 & (3) & 0.031 & 237363.672 & 0.419 & (2) & 0.069 & 258411.000 & 0.000 & (3) & \textbf{0.009} & 237056.631 & 0.443\\
stein27 & (2) & \textbf{0.183} & 18.000 & 0.000 & (2) & 0.188 & 18.000 & 0.000 & (3) & 4.341 & 15.400 & 0.300\\
stein45 & (2) & \textbf{4.434} & 30.000 & 0.000 & (2) & 5.410 & 30.000 & 0.000 & TL & 19.999 & -- & --\\
lseu & (3) & 0.021 & 1079.438 & 0.217 & (3) & 0.079 & 1107.827 & 0.250 & (3) & \textbf{0.003} & 956.641 & 0.312\\
egout & (2) & 0.016 & 568.101 & 0.000 & (2) & \textbf{0.012} & 568.101 & 0.000 & TL & 20.001 & -- & --\\
flugpl & (2) & \textbf{0.004} & 1201500.000 & 0.000 & (2) & 0.005 & 1201500.000 & 0.000 & (3) & 0.005 & 1169362.271 & 0.371\\
noswot & (3) & \textbf{0.012} & -43.000 & 0.177 & (3) & 0.014 & -43.000 & 0.177 & (3) & 0.478 & -43.000 & 0.360\\
blend2 & (3) & 0.039 & 7.112 & 0.237 & (3) & 0.233 & 7.730 & 0.172 & (3) & \textbf{0.005} & 6.927 & 0.249\\
bell3a & (2) & 6.031 & 879218.716 & 0.000 & (2) & 2.536 & 879218.716 & 0.000 & (3) & \textbf{0.631} & 862963.002 & 0.300\\
mas74 & (3) & 0.273 & 10608.034 & 0.437 & (3) & 0.044 & 10524.444 & 0.432 & (3) & \textbf{0.005} & 10516.387 & 0.449\\
mas76 & (3) & 0.144 & 39244.839 & 0.337 & (3) & 0.253 & 39179.162 & 0.338 & (3) & \textbf{0.010} & 38934.230 & 0.352\\
markshare1 & (3) & 0.007 & 0.000 & 0.313 & (3) & 0.008 & 0.000 & 0.313 & (3) & \textbf{0.003} & 0.000 & 0.430\\
markshare2 & (3) & 0.042 & 0.000 & 0.328 & (3) & 0.080 & 0.000 & 0.324 & (3) & \textbf{0.003} & 0.000 & 0.374\\
misc03 & (3) & 0.043 & 2145.000 & 0.400 & (3) & \textbf{0.031} & 2140.000 & 0.444 & (3) & 0.296 & 1945.250 & 0.450\\
misc06 & (3) & 0.069 & 12847.056 & 0.285 & (3) & 0.132 & 12848.473 & 0.240 & (3) & \textbf{0.026} & 12841.948 & 0.287\\
misc07 & (3) & 0.094 & 1425.000 & 0.429 & (3) & \textbf{0.081} & 1425.000 & 0.429 & (3) & 0.489 & 1430.286 & 0.450\\
enigma & (3) & \textbf{0.011} & 0.000 & 0.048 & (3) & 0.016 & 0.000 & 0.003 & TL & 20.002 & -- & --\\
gt2 & TL & 20.024 & -- & -- & (2) & \textbf{0.029} & 25733.000 & 0.000 & (3) & 0.053 & 13819.774 & 0.160\\
vpm1 & (3) & 0.030 & 17.489 & 0.333 & (2) & 0.020 & 20.000 & 0.000 & (3) & \textbf{0.004} & 15.517 & 0.450\\
vpm2 & (3) & 0.014 & 11.136 & 0.446 & (3) & 0.014 & 11.136 & 0.446 & (3) & \textbf{0.005} & 9.889 & 0.450\\
pp08a & (3) & 0.121 & 3156.917 & 0.320 & (3) & 0.682 & 2933.314 & 0.350 & (3) & \textbf{0.004} & 2920.845 & 0.360\\
qiu & TL & 20.003 & -- & -- & TL & 20.002 & -- & -- & (3) & \textbf{1.600} & -848.252 & 0.318\\
danoint & (3) & 0.313 & 62.690 & 0.424 & (3) & 0.339 & 62.756 & 0.434 & (3) & \textbf{0.035} & 62.650 & 0.446\\
fixnet6 & (2) & 0.263 & 4092.000 & 0.000 & (2) & \textbf{0.125} & 4092.000 & 0.000 & (3) & 9.103 & 1484.743 & 0.092\\
modglob & TL & 20.018 & -- & -- & (2) & \textbf{0.105} & 20740508.086 & 0.000 & (3) & 1.592 & 20435196.792 & 0.440\\
set1ch & TL & 20.024 & -- & -- & (2) & 0.277 & 54537.750 & 0.000 & (3) & \textbf{0.008} & 32019.009 & 0.398\\
\hline
\end{tabular}
\endgroup
\caption{Numerical results for the restoration phase with $\kappa_\omega=1000$, $\beta=0.9$, and $\gamma=0.1$.}\label{tab:1}
\end{table}

We observe that the (CD) method returns only filter-admissible points, in comparison to the (BB) and (BC) variants, which return \hyperlink{or2}{(OR2)} and \hyperlink{or3}{(OR3)} in more balanced rates. Restricting each pairwise comparison to problems where both methods return \hyperlink{or3}{(OR3)}, BF-CD is faster on 13 of 16 problems against BF-BB and 10 of 13 against BF-BC, with a median elapsed time of just $0.010$ s, which is partly (16/24 cases) reflected upon its early inexpensive $2p$ LPs test. Moreover, there is a visible tradeoff on the returned objective value and integrality violation: In most cases where BF-CD and either BF-BB or BF-BC return \hyperlink{or3}{(OR3)}, the cell-deletion variation returns an admissible point that has higher $h$ violation but better objective reduction. Overall, (CD) provides inexpensive fractional progress and fast admissible-point detection, whereas the standard branching alternatives more frequently complete the node's integer optimization.

\subsection{Optimality experiments}

Although the proposed cell-deletion method tends to return a filter–funnel-admissible point faster than the bound and cut implementations considered here, its emphasis on early admissibility does not prioritize finding a node-optimal solution, the latter of which can substantially accelerate convergence to optimality. Moreover, BF-CD can spend considerable time in its full iterative stage when the auxiliary disjunctive subproblems become expensive, which may in turn cause a large number of restoration calls. These observations motivate a hybrid approach that combines the early-return capability of CD with a branch-and-bound fallback. The introduced hybrid method (which we still denote BF-CD for simplicity of exposition) first tries (CD) for at most 1 sec. per restoration call, subject to a cumulative allowance $\min\{3 \text{ secs}, 0.15T\}$ for run limit $T$. If (CD) exhausts an allowance or a resource limit, the branch-and-bound method of BF-BB completes the request with the remaining budget and handles all subsequent restoration calls in that run within their respective selected nodes. Also, in comparison to the previous subsection, the recovery mechanism of Algorithm \ref{alg:recovery} is enabled for this hybrid version, with the aforementioned optimality threshold of $10^{-6}$.

The set consists of ten selected mknapcb1 instances \cite{chu1998knapsack}. Besides studying the numerical performance of (BF) with the introduced metric $g$, we also study the performance of the objective function $f$, which is the standard metric choice in classical filter methods in nonconvex constrained optimization, and which shares similar theoretical guarantees with the former in this paper (see Remark \ref{rem:f1}). We highlight the objective metric used at the end each variant's abbreviation (e.g., we write BF-BB-f and BF-BB-g to distinguish the two). As a reference solver we compare the six methods against SCIP 10.0.0, which uses SoPlex 8.0.0 through PySCIPOpt 6.0.0 \cite{SCIPOptSuite10}. SCIP receives the same prepared input with (BF) and retains its native presolve, propagation, separation, branching, node selection and primal heuristics. 

\begin{table}[H]
\centering
\begingroup
\fontsize{9}{10.5}\selectfont
\setlength{\tabcolsep}{3.5pt}
\renewcommand{\arraystretch}{1.05}
\begin{tabular}{@{}lrrrrrrr@{}}
\toprule
Problem & BF-CD-$f$ & BF-CD-$g$ & BF-BC-$f$ & BF-BC-$g$ & BF-BB-$f$ & BF-BB-$g$ & SCIP\\
\midrule
mknapcb1\_01 & 21.568 & 87.250 & 23.588 & 12.654 & 9.629 & 9.274 & \textbf{8.026}\\
mknapcb1\_02 & 14.406 & 50.631 & 21.357 & 12.917 & 11.143 & 9.082 & \textbf{7.957}\\
mknapcb1\_03 & 12.866 & 50.127 & 17.165 & 13.314 & 9.819 & \textbf{5.533} & 19.324\\
mknapcb1\_11 & 10.272 & 36.742 & 16.876 & 11.123 & 8.091 & 4.625 & \textbf{2.472}\\
mknapcb1\_12 & 14.417 & 64.551 & 20.400 & 30.337 & 15.691 & 15.862 & \textbf{5.600}\\
mknapcb1\_14 & 17.313 & 55.058 & 27.168 & 13.180 & 17.219 & 11.319 & \textbf{10.488}\\
mknapcb1\_21 & 4.889 & 5.667 & 3.648 & 3.111 & 2.418 & 3.556 & \textbf{2.272}\\
mknapcb1\_22 & 7.173 & 15.723 & 7.840 & 7.825 & 9.569 & 6.258 & \textbf{1.674}\\
mknapcb1\_23 & 11.790 & 32.120 & 17.296 & 6.109 & 6.418 & \textbf{5.611} & 7.874\\
mknapcb1\_24 & 14.657 & 57.174 & 21.462 & 10.439 & 13.755 & \textbf{7.191} & 9.487\\
\bottomrule
\end{tabular}
\endgroup
\caption{Optimality times of the six (BF) variants and SCIP, with $\kappa_\omega=1000$, $\beta=0.9$, $\gamma=0.1$.}
\end{table}

In the first parameter setting which matches that of the restoration numerical experiment, the objective function constitutes a much better filter metric when it comes to the (CD) variation, but falls short against $g$ in almost all instances regarding the (BB) and (BC) implementations. However, SCIP tends to outperform all branch-and-filter methods. In addition, the cumulative outer iteration count of BF-CD-g is equal to 1706 (1151 \hyperlink{or3}{(OR3)} outcomes), which is substantially larger (almost ten times larger) than that of the other methods. For comparison, BF-CD-$f$ produces $167$ \hyperlink{or3}{(OR3)} returns, performs $188$ outer iterations, and produces more node-optimal \hyperlink{or2}{(OR2)} returns that the $g$ variant ($23$ versus $10$). This contrast is consistent with Example \ref{ex:ex}, which shows that $g$ can accept fractional points that would be otherwise rejected by $f$.

In combination with the results of Table 1, the above observations indicate that more filter-admissible returns are not automatically beneficial, as they can prolong the sequence before node completion and an approximate global certificate. Since a larger $\kappa_\omega$ enlarges the funnel and a larger $\beta$ and smaller $\gamma$ steepen the slanted filter boundary, a direct way to mitigate the number of \hyperlink{or3}{(OR3)} returns while still remaining informative is to tune the parameters appropriately, by reducing $\kappa_\omega$ and $\beta$, and increasing $\gamma$. This modification is also strongly encouraged by the complexity bound of Theorem \ref{th:complexity}, as it reduces the number of total outer iterations of the algorithm, which may then lead to a reduction in the number of oracle calls.

\begin{table}[h]
    \centering
   \begingroup
\fontsize{9}{10.5}\selectfont
\setlength{\tabcolsep}{3.5pt}
\renewcommand{\arraystretch}{1.05}
\begin{tabular}{@{}lrrrrrrr@{}}
\toprule
Problem & BF-CD-$f$ & BF-CD-$g$ & BF-BC-$f$ & BF-BC-$g$ & BF-BB-$f$ & BF-BB-$g$ & SCIP\\
\midrule
mknapcb1\_01 & 8.011 & 8.376 & 6.361 & 6.505 & 4.042 & \textbf{3.797} & 7.994\\
mknapcb1\_02 & 3.017 & 3.137 & 3.089 & 3.132 & 1.923 & \textbf{1.864} & 8.127\\
mknapcb1\_03 & 5.477 & 5.262 & 5.047 & 4.900 & 3.274 & \textbf{3.066} & 18.643\\
mknapcb1\_11 & 4.629 & 3.703 & 2.451 & 2.749 & \textbf{1.199} & 1.305 & 2.632\\
mknapcb1\_12 & 7.416 & 10.173 & 2.449 & 2.506 & 1.691 & \textbf{1.631} & 5.543\\
mknapcb1\_14 & 5.056 & 4.975 & 6.130 & 5.846 & \textbf{3.261} & 3.479 & 10.011\\
mknapcb1\_21 & 2.218 & 2.268 & 1.531 & 1.593 & 1.191 & \textbf{1.177} & 2.253\\
mknapcb1\_22 & 2.347 & 5.527 & 1.010 & 0.997 & \textbf{0.742} & 0.773 & 1.825\\
mknapcb1\_23 & 4.165 & 3.516 & 2.500 & 2.542 & \textbf{1.560} & 1.679 & 6.159\\
mknapcb1\_24 & 5.632 & 4.174 & 2.570 & 2.700 & 1.969 & \textbf{1.884} & 8.538\\
\bottomrule
\end{tabular}
\endgroup
    \caption{Optimality times of the six (BF) variants and SCIP, with $\kappa_\omega=100$, $\beta=0.4$, $\gamma=0.6$.}
    \label{tab:placeholder}
\end{table}
Following the discussion above, the second choice of parameters leads to a performance enhancement of (BF), especially with respect to the proposed (CD) variant over SCIP. The objective function still remains competitive against $g$, but the latter performs better in most instances over the (BB) variant. Within BF-CD, $g$ is faster than $f$ on half the problems and has a lower median time, but a higher mean, largely because of the additional work on mknapcb1\_12 and mknapcb1\_22 which account for most of the \hyperlink{or3}{(OR3)} outcomes. Further, (CD) terminates faster than the (BC) variant in a few occasions, while both BF-BB variants outperform SCIP on every single problem. Interestingly, all BF-BB and BF-BC methods now terminate after at most two outer iterations for each knapsack problem, i.e, each method performs at most two restoration iterations - in comparison to approximately 17 on average in the previous experimental run - returns almost none \hyperlink{or3}{(OR3)} outcomes (with an exception of BF-BB at mknapcb1\_21, which returned one such outcome), and then outputs an optimal solution.

\subsection{Time to approximate solutions}

Due to the implementation of the restoration oracle, for the last parameter setting above, BF-BB and BF-BC essentially represent classical branch-and-bound and branch-and-cut methods, respectively. However, they should not be confused with the standard Gurobi native (BB)/(BC) methods: (BF) branches externally, selects a child, and then solves successive node-restricted Gurobi models. Furthermore, restoration calls solve separate subproblems and perform additional fractional-admissibility callbacks. The two Gurobi native branching processes instead search the full prepared model in one solve instance. Therefore, the initial restricted problem, the search order, and the retained search states are substantially different, even when BF-BB/BF-BC return no \hyperlink{or3}{(OR3)} outcome during the restoration phase.

This distinction motivates a more detailed analysis of the performance differences of (BB) and (BC) between Gurobi and the proposed filter search mechanism. Here, we examine the maximum time needed to identify an approximate optimal solution. For scaling purposes, the notion of approximation is slightly different to that of Definition \ref{def:optimal}, in the sense that we are now requiring a a bound with respect to a relative objective error. More specifically, for problem $i$ and method $M$, let $U_{i,t}^M$ represent the best validated observed incumbent in minimization form at time $t$, and $z^*_i$ be the common verified optimal value, determined retrospectively. Given optimality threshold $\varepsilon\geq 0$, we define
\begin{equation*}
    T_{i,\varepsilon}^M:=\inf\left\{t\geq 0: U_{i,t}^M-z^*_i\leq \varepsilon\max\{1,|z^*_i|\}\right\}.
\end{equation*}
We examine the time allowance required to attain a prescribed objective accuracy throughout the benchmark panel, via the Gurobi and (BF) branching variants. More specifically, for each method, we study the quantity $\tau_\varepsilon:=\max_{i\in \{1,...,10\}} T_{i,\varepsilon}$, which represents the smallest common per-problem time budget within which every tested instance reaches the prescribed target.

In the plots below, we study the ratio $\tau_\varepsilon^{\text{Gurobi}}/\tau_\varepsilon^{\text{BF}}$ for optimality error $\varepsilon\in [0,0.01]$. A ratio above 1 indicates that (BF) has a smaller maximum time, over the ten-problem batch, to first finding a feasible solution within the prescribed optimality error. For consistency purposes, at $\varepsilon=0$ the target is first discovery of an optimal incumbent within the predetermined optimality tolerance $10^{-6}$, meaning that, during implementation, the right-hand side of the error bound in the definition of $T_{i,\varepsilon}^M$ equals $\max\{10^{-6},\varepsilon\max\{1,|f^*|\}\}$. We analyze the ratio for Gurobi-BB/Gurobi-BC against the BF-BB/BF-BC variants, respectively. Here, the Gurobi-BB and Gurobi-BC native references use the identical prepared inputs and corresponding native restoration policies (cuts disabled/enabled, respectively), with exactly the same heuristic, presolve and numerical settings as the (BF) methods.
\begin{figure}[htbp]
    \centering
    \includegraphics[width=\textwidth]{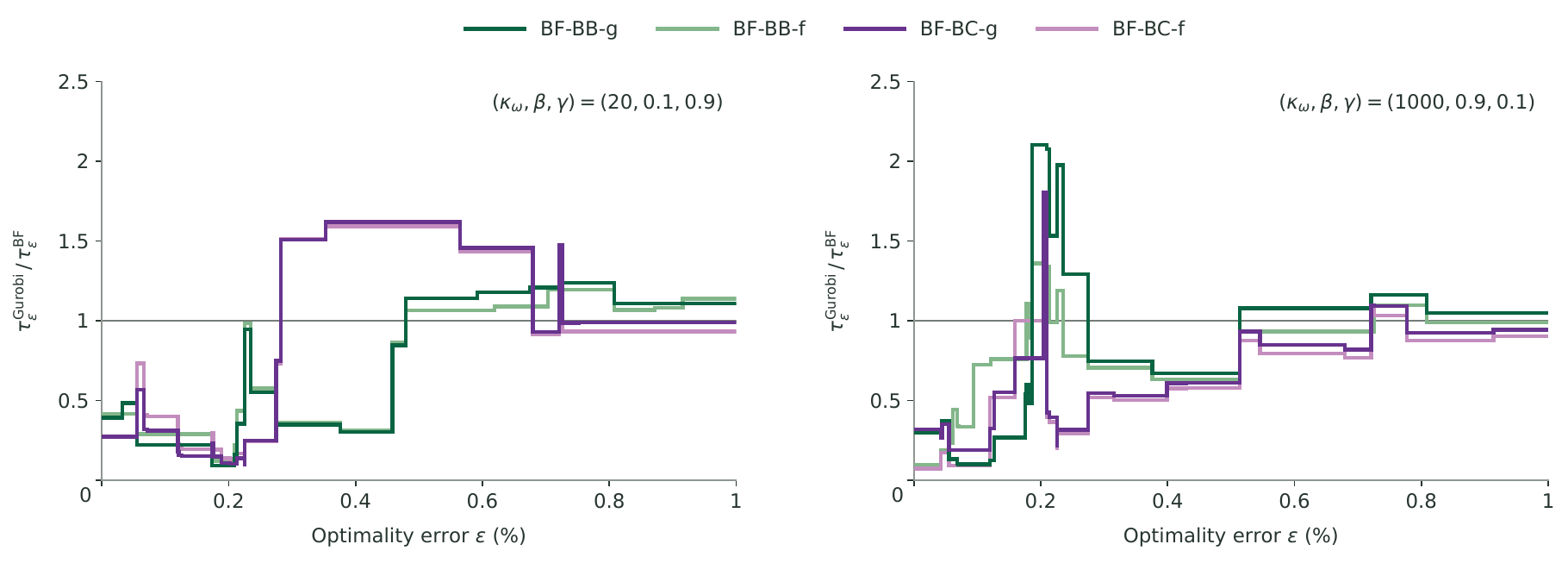}
    \caption{Ratio $\tau_{\varepsilon}^{\mathrm{Gurobi}}/
    \tau_{\varepsilon}^{\mathrm{BF}}$ as a function of the optimality
    error $\varepsilon\in [0,0.01]$ over the ten
    knapsack problems. Each (BF) variant is compared with its
    corresponding Gurobi (BB) or (BC) reference method (e.g, the purple trajectory corresponds to the ratio $\tau_{\varepsilon}^{\text{Gurobi-BC}}/\tau_{\varepsilon}^{\text{BF-BC-g}}$). A ratio above 1 translates to (BF) finding an approximate solution faster than Gurobi.}
    \label{fig:tau}
\end{figure}

In the first experimental run (left plot of Figure \ref{fig:tau}), the BF-BB variants give a $\tau_\varepsilon$ ratio above 1 approximately in the optimality range $\varepsilon\in[0.005,0.01]$, which means that BF-BB discovers an approximate solution earlier than the native Gurobi branch-and-bound method. Moreover, BF-BC-g lowers $\tau_\varepsilon$ at $0.5\%$ error from 212.356 to 131.066 ms, which represents a 1.620 ratio and 38.3\% reduction compared to Gurobi-BC. This advantage extends in an approximate error interval $0.3\%-0.7\%$. It should be noted that different problems determine the (BF) maximum as \(\varepsilon\) varies; within the favorable regions, BF-BB-g passes among mknapcb1\_12, \_02, \_01, and \_22, while BF-BC-g switches between \_14 and \_02. Crucially, at $\varepsilon=0.005$, removing any single problem from both compared methods preserves the advantage of BF-BC-g and BF-BC-f, with minimum recomputed ratios of 1.111 and 1.165, respectively, compared with 1.620 and 1.593 for the complete batch. At the same tolerance, BF-BB-g retains a ratio above 1 in nine of the ten single-problem removals; the only exception is mknapcb1\_01, whose removal reduces the ratio from 1.140 to 0.897. 

Similarly to the previous $(100,0.4,0.6)$ run, in the tighter $(20,0.1,0.9)$ setting, both (BF) methods output no \hyperlink{or3}{(OR3)} outcomes during restoration and terminate only after one outer iteration, meaning that they represent classical (BB) and (BC) methods incorporated with an initial filter search strategy. The unique characteristics of the filter mechanism become more evident in the second parameter setting (right plot of Figure \ref{fig:tau}), where restoration is a far more active part of (BF)'s performance. Across the full run, all
ten knapsack problems produce several admissible restoration returns under BF-BB-g: 175 \hyperlink{or3}{(OR3)} outcomes and 195 outer iterations in total. At the same time, BF-BB-g outperforms Gurobi-BB in the interval $0.52\%-1\%$, and is the only one to do so in such a large interval across the four variants. In addition, a significant ratio spike for both methods is observed in the interval $0.19\%-0.25\%$. Interestingly, at $\varepsilon=0.002$, the batch maximum and mean target times are
respectively 52.5\% and 20.8\% lower than those of the matched Gurobi-BB reference.

It should be mentioned that Gurobi converges to an optimal solution (within error tolerance $10^{-6}$) faster than (BF) in both parameter settings, which is reflected upon the $\tau_\varepsilon$ ratio for $\varepsilon=0$ in Figure \ref{fig:tau}. Therefore, the demonstrated strength of (BF) in this problem batch is competitive access to useful solution quality, rather than convergence to absolute optimality, while including a regime where restoration activity is sustainable.

\section{Discussion}\label{sec:discussion}

We have introduced a new iterative algorithm that directs MILP search through explicit acceptance conditions on fractional relaxation points. Under mild assumptions, the proposed method yields global convergence and outer iteration complexity guarantees. This provides a principled basis for investigating when inexpensive fractional progress can effectively replace the direct completion of a node’s integer optimization. The introduced filter framework is simple and follows from classical filter methods in nonconvex optimization \cite{fletcher2002nonlinear}, thus leaving several components open to further development. First, the initial branching step can be complemented by explicit cut generation and domain propagation within the outer algorithm. Such operations can strengthen relaxations and reduce the remaining search, as illustrated by classical (BC) methods \cite{balas1996mixed}, which may then lead to the development of more efficient \textit{branch-cut-and-filter} algorithms for MILPs. Second, one could adapt other filter mechanisms to the proposed framework, such as nonmonotone variants that permit a limited number of failed comparisons with existing entries \cite{shen2012nonmonotone}.  However, the transition of our theoretical warranties to more complicated filter settings must be rigorously verified. Further, the success of the filter mechanism in establishing global optimality should be tested against other settings of mixed-integer optimization, such as in convex MINLP, where continuous relaxations and polyhedral outer approximations provide valid lower bounds \cite{bonami2008algorithmic}. Finally, the filter acceptance rule should be explicitly studied as a node-selection component within a common solver environment, against the state-of-the-art best-bound, depth-first, and hybrid search strategies \cite{linderoth1999computational}.

The counter-intuitive definition of $g$, which promotes increasing the objective value rather than decreasing it, seems to perform better in settings where (BF) becomes competitive against benchmark solvers, and admits points that would otherwise not be acceptable by an $f$-filter (Example \ref{ex:ex}), despite the fact that it admits a worse outer iteration complexity bound than $f$. An analytical study of the behavior of the two objective measures is highly encouraged, in order to clarify their respective advantages and limitations. Such an analysis may also identify alternative measures that better reflect the distinct nature of mixed-integer programming: In comparison to nonconvex continuous optimization, where progress involves both improving feasible solutions and strengthening relaxation bounds, an increase in a node's optimal relaxation value can be beneficial even when it does not improve the incumbent. 

Besides exploring alternative optimality metrics for (BF), another future direction involves the design of other optimality mechanisms. We have shown that the filter is responsible for driving the sequence of generated iterates to integrality, while the funnel mechanism ensures convergence to global optimality. In classical filter methods, convergence to first-order stationary points is established by combining filter acceptance with suitable step-generation procedures and sufficient decrease conditions within trust-region or line-search frameworks \cite{fletcher2002global,wachter2005line}.  Investigating alternative mechanisms that play an analogous role in (BF), while preserving its global optimality guarantees, is therefore another promising direction.

Although branch-and-filter is accompanied by some promising computational results, there is plenty of room for improvement. First, the restoration experiments identify a useful computational tradeoff: The proposed cell-deletion
method often returns a fractional filter-admissible point quickly, whereas the tested (BB) and (BC) implementations more frequently complete the selected node’s integer
optimization. However, repeated auxiliary disjunctive solves can make the full (CD) stage expensive. A hybrid modification of (CD) with a (BB) fallback accelerates convergence and indicates, along with the computational efficiency of the other two variations, that a more balanced search between fractional filter-admissible and node-optimal points can reduce computational work significantly. In addition, several improvements and heuristics could make the restoration phase more efficient. For instance, information from earlier calls could be reused to warm-start subproblems, guide the initial inexpensive tests, or identify large groups of cells that can be deleted simultaneously. The time spent on cell deletion before switching to branching could also adapt to its recent performance, the dimensions of the inner subproblems, and the progress made in improving bounds. Another further direction is to extend the convergence analysis to inexact LP solves, allowing controlled numerical errors while preserving the validity of the introduced acceptance tests, relaxation bounds, and restoration certificates.

A future computational study should also focus on parameter tuning. As of now, there is a tradeoff between widening the filter and funnel sets, thus expanding the search for admissible iterates, and tightening the same acceptance sets, leading to fewer outer iterations. This mathematical observation is supported by the conducted numerical experiments, and the effect of each parameter on the final convergence time requires a deeper sensitivity analysis. Furthermore, the worst-case time-to-target results over the tested panel indicate that our introduced method is able to identify approximate optimal solutions faster than the state-of-the-art solver Gurobi in the selected problem batch, under standard (BB) and (BC) implementations of the restoration oracle. These numerical results motivate further study of intermediate solution quality, alongside total solution time and certificate production. Finally, the filter's role in reducing integrality violation motivates investigating (BF) as a method for finding a first feasible solution, under simple modifications (Section \ref{subsec:feasibility}). A dedicated study should compare it with classical feasibility pump methods, which alternate between integer rounding and optimization over the continuous relaxation \cite{fischetti2005feasibility}. Overall, a much more thorough computational study is required to determine absolute efficiency of the proposed (BF) method.

In conclusion, our numerical experiments indicate that a significant part of the computational work is performed within the restoration phase, making its improvement a central direction for future research. The theoretical warranties provided by the proposed filter device also provide a strong motivation for pursuing this direction. We hope that this novel use of classical filter methods inside branching procedures will encourage further theoretical and computational developments, and that an efficient implementation of the introduced restoration oracle could help realize the potential of branch-and-filter as a competitive method for generic MILP problems.

\begingroup

\section*{Use of AI Statement}

The use of ChatGPT-Pro (GPT-6 Astra) is acknowledged by the author in generating the figures and improving the final exposition of the tables. The same model and Codex assisted with the implementation of the numerical experiments (Section \ref{sec:experiments}). The author takes full responsibility for the veracity of all theoretical claims and technical details in this paper.

\raggedright
\bibliography{references}

@article{FletcherLeyfferToint2002,
  author = {Fletcher, Roger and Leyffer, Sven and Toint, Philippe L.},
  title = {On the Global Convergence of a Filter--{SQP} Algorithm},
  journal = {SIAM Journal on Optimization},
  volume = {13}, number = {1}, pages = {44--59}, year = {2002}
}

@article{Hoffman1952,
  author = {Hoffman, Alan J.},
  title = {On Approximate Solutions of Systems of Linear Inequalities},
  journal = {Journal of Research of the National Bureau of Standards},
  volume = {49}, number = {4}, pages = {263--265}, year = {1952}
}

@article{sun2024decomposition,
  title={Decomposition methods for global solution of mixed-integer linear programs},
  author={Sun, Kaizhao and Sun, Mou and Yin, Wotao},
  journal={SIAM Journal on Optimization},
  volume={34},
  number={2},
  pages={1206--1235},
  year={2024},
  publisher={SIAM}
}

@book{Schrijver1986,
  author = {Schrijver, Alexander},
  title = {Theory of Linear and Integer Programming},
  publisher = {John Wiley \& Sons}, address = {Chichester}, year = {1986}
}

@article{feizollahi2017exact,
  title={Exact augmented Lagrangian duality for mixed integer linear programming},
  author={Feizollahi, Mohammad Javad and Ahmed, Shabbir and Sun, Andy},
  journal={Mathematical Programming},
  volume={161},
  number={1},
  pages={365--387},
  year={2017},
  publisher={Springer}
}

@article{dash2018binary,
  title={Binary extended formulations of polyhedral mixed-integer sets},
  author={Dash, Sanjeeb and G{\"u}nl{\"u}k, Oktay and Hildebrand, Robert},
  journal={Mathematical Programming},
  volume={170},
  number={1},
  pages={207--236},
  year={2018},
  publisher={Springer}
}

@article{testa2019distributed,
  title={Distributed mixed-integer linear programming via cut generation and constraint exchange},
  author={Testa, Andrea and Rucco, Alessandro and Notarstefano, Giuseppe},
  journal={IEEE Transactions on Automatic Control},
  volume={65},
  number={4},
  pages={1456--1467},
  year={2019},
  publisher={IEEE}
}

@article{chin2003global,
  title={On the global convergence of an SLP--filter algorithm that takes EQP steps},
  author={Chin, Choong Ming and Fletcher, Roger},
  journal={Mathematical Programming},
  volume={96},
  number={1},
  pages={161--177},
  year={2003},
  publisher={Springer}
}

@article{fletcher2002nonlinear,
  title={Nonlinear programming without a penalty function},
  author={Fletcher, Roger and Leyffer, Sven},
  journal={Mathematical programming},
  volume={91},
  number={2},
  pages={239--269},
  year={2002},
  publisher={Springer}
}

@article{gould2010nonlinear,
  title={Nonlinear programming without a penalty function or a filter},
  author={Gould, Nicholas IM and Toint, Ph L},
  journal={Mathematical Programming},
  volume={122},
  number={1},
  pages={155--196},
  year={2010},
  publisher={Springer}
}

@misc{nocedal2006numerical,
  title={Numerical Optimization},
  author={Nocedal, J and Wright, S},
  year={2006},
  publisher={Springer}
}

@incollection{land2009automatic,
  title={An automatic method for solving discrete programming problems},
  author={Land, Ailsa H and Doig, Alison G},
  booktitle={50 Years of Integer Programming 1958-2008: From the Early Years to the State-of-the-Art},
  pages={105--132},
  year={2009},
  publisher={Springer}
}

@article{balas1996mixed,
  title={Mixed 0-1 programming by lift-and-project in a branch-and-cut framework},
  author={Balas, Egon and Ceria, Sebasti{\'a}n and Cornu{\'e}jols, G{\'e}rard},
  journal={Management Science},
  volume={42},
  number={9},
  pages={1229--1246},
  year={1996},
  publisher={INFORMS}
}

@article{barnhart1998branch,
  title={Branch-and-price: Column generation for solving huge integer programs},
  author={Barnhart, Cynthia and Johnson, Ellis L and Nemhauser, George L and Savelsbergh, Martin WP and Vance, Pamela H},
  journal={Operations research},
  volume={46},
  number={3},
  pages={316--329},
  year={1998},
  publisher={INFORMS}
}

@article{vanderbeck1996exact,
  title={An exact algorithm for IP column generation},
  author={Vanderbeck, Fran{\c{c}}ois and Wolsey, Laurence A},
  journal={Operations research letters},
  volume={19},
  number={4},
  pages={151--159},
  year={1996},
  publisher={Elsevier}
}

@article{balas1993lift,
  title={A lift-and-project cutting plane algorithm for mixed 0--1 programs},
  author={Balas, Egon and Ceria, Sebasti{\'a}n and Cornu{\'e}jols, G{\'e}rard},
  journal={Mathematical programming},
  volume={58},
  number={1},
  pages={295--324},
  year={1993},
  publisher={Springer}
}

@article{chen2011finite,
  title={Finite disjunctive programming characterizations for general mixed-integer linear programs},
  author={Chen, Binyuan and K{\"u}{\c{c}}{\"u}kyavuz, Simge and Sen, Suvrajeet},
  journal={Operations Research},
  volume={59},
  number={1},
  pages={202--210},
  year={2011},
  publisher={INFORMS}
}

@article{benders1962partitioning,
  title={Partitioning procedures for solving mixed-variables programming problems},
  author={Benders, Jacques F},
  journal={Numerische mathematik},
  volume={4},
  number={1},
  pages={238--252},
  year={1962},
  publisher={Springer-Verlag Berlin/Heidelberg}
}

@article{balas1967discrete,
  title={Discrete programming by the filter method},
  author={Balas, Egon},
  journal={Operations Research},
  volume={15},
  number={5},
  pages={915--957},
  year={1967},
  publisher={INFORMS}
}

@article{leyffer2001integrating,
  title={Integrating SQP and branch-and-bound for mixed integer nonlinear programming},
  author={Leyffer, Sven},
  journal={Computational optimization and applications},
  volume={18},
  number={3},
  pages={295--309},
  year={2001},
  publisher={Springer}
}

@article{abhishek2010filmint,
  title={FilMINT: An outer approximation-based solver for convex mixed-integer nonlinear programs},
  author={Abhishek, Kumar and Leyffer, Sven and Linderoth, Jeff},
  journal={INFORMS Journal on computing},
  volume={22},
  number={4},
  pages={555--567},
  year={2010},
  publisher={INFORMS}
}

@inproceedings{fernandes2014branch,
  title={Branch and bound based coordinate search filter algorithm for nonsmooth nonconvex mixed-integer nonlinear programming problems},
  author={Fernandes, Florbela P and Costa, M Fernanda P and Fernandes, Edite MGP},
  booktitle={International Conference on Computational Science and Its Applications},
  pages={140--153},
  year={2014},
  organization={Springer}
}

@article{abramson2004filter,
  title={Filter pattern search algorithms for mixed variable constrained optimization problems},
  author={Abramson, Mark A and Audet, Charles and Dennis Jr, John E},
  year={2004}
}

@article{balas1979disjunctive,
  title={Disjunctive programming},
  author={Balas, Egon},
  journal={Annals of discrete mathematics},
  volume={5},
  pages={3--51},
  year={1979},
  publisher={Elsevier}
}

@article{wachter2005line,
  title={Line search filter methods for nonlinear programming: Motivation and global convergence},
  author={W{\"a}chter, Andreas and Biegler, Lorenz T},
  journal={SIAM Journal on Optimization},
  volume={16},
  number={1},
  pages={1--31},
  year={2005},
  publisher={SIAM}
}

@article{fletcher2002global,
  title={Global convergence of a trust-region SQP-filter algorithm for general nonlinear programming},
  author={Fletcher, Roger and Gould, Nicholas IM and Leyffer, Sven and Toint, Philippe L and W{\"a}chter, Andreas},
  journal={SIAM Journal on Optimization},
  volume={13},
  number={3},
  pages={635--659},
  year={2002},
  publisher={SIAM}
}

@article{linderoth1999computational,
  title={A computational study of search strategies for mixed integer programming},
  author={Linderoth, Jeff T and Savelsbergh, Martin WP},
  journal={INFORMS Journal on Computing},
  volume={11},
  number={2},
  pages={173--187},
  year={1999},
  publisher={INFORMS}
}

@article{boland2015augmented,
  title={On the augmented Lagrangian dual for integer programming},
  author={Boland, Natashia L and Eberhard, Andrew C},
  journal={Mathematical Programming},
  volume={150},
  number={2},
  pages={491--509},
  year={2015},
  publisher={Springer}
}

@article{belotti2013mixed,
  title={Mixed-integer nonlinear optimization},
  author={Belotti, Pietro and Kirches, Christian and Leyffer, Sven and Linderoth, Jeff and Luedtke, James and Mahajan, Ashutosh},
  journal={Acta Numerica},
  volume={22},
  pages={1--131},
  year={2013},
  publisher={Cambridge University Press}
}

@article{fischetti2010heuristics,
  title={Heuristics in mixed integer programming},
  author={Fischetti, Matteo and Lodi, Andrea and others},
  journal={Wiley Encyclopedia of Operations Research and Management Science. John Wiley \& Sons, Inc},
  pages={2--23},
  year={2010}
}

@article{san2022new,
  title={A new branch-and-filter exact algorithm for binary constraint satisfaction problems},
  author={San Segundo, Pablo and Furini, Fabio and Le{\'o}n, Rafael},
  journal={European Journal of Operational Research},
  volume={299},
  number={2},
  pages={448--467},
  year={2022},
  publisher={Elsevier}
}

@misc{gurobi2026,
  author = {{Gurobi Optimization, LLC}},
  title  = {{Gurobi Optimizer Reference Manual}},
  year   = {2026},
  url    = {https://www.gurobi.com}
}

@techreport{SCIPOptSuite10,
  author = {Christopher Hojny and Mathieu Besan{\c{c}}on
    and Ksenia Bestuzheva and Sander Borst and Antonia Chmiela
    and Jo{\~{a}}o Dion{\'{i}}sio and Leon Eifler
    and Mohammed Ghannam and Ambros Gleixner
    and Adrian G{\"o}{\ss} and Alexander Hoen
    and Rolf van der Hulst and Dominik Kamp and Thorsten Koch
    and Kevin Kofler and Jurgen Lentz and Stephen J. Maher
    and Gioni Mexi and Erik M{\"u}hmer and Marc E. Pfetsch
    and Sebastian Pokutta and Felipe Serrano and Yuji Shinano
    and Mark Turner and Stefan Vigerske and Matthias Walter
    and Dieter Weninger and Liding Xu},
  title       = {{The SCIP Optimization Suite 10.0}},
  type        = {Technical Report},
  institution = {Optimization Online},
  month       = nov,
  year        = {2025},
  url         = {https://optimization-online.org/2025/11/the-scip-optimization-suite-10-0/}
}

@article{bixby1998miplib,
  author  = {Bixby, Robert E. and Ceria, Sebastian
    and McZeal, Cassandra M. and Savelsbergh, Martin W. P.},
  title   = {An Updated Mixed Integer Programming Library:
    {MIPLIB} 3.0},
  journal = {Optima},
  volume  = {58},
  pages   = {12--15},
  year    = {1998}
}

@article{chu1998knapsack,
  author  = {Chu, P. C. and Beasley, J. E.},
  title   = {A Genetic Algorithm for the Multidimensional
    Knapsack Problem},
  journal = {Journal of Heuristics},
  volume  = {4},
  number  = {1},
  pages   = {63--86},
  year    = {1998}
}

@article{bonami2008algorithmic,
  title={An algorithmic framework for convex mixed integer nonlinear programs},
  author={Bonami, Pierre and Biegler, Lorenz T and Conn, Andrew R and Cornu{\'e}jols, G{\'e}rard and Grossmann, Ignacio E and Laird, Carl D and Lee, Jon and Lodi, Andrea and Margot, Fran{\c{c}}ois and Sawaya, Nicolas and others},
  journal={Discrete optimization},
  volume={5},
  number={2},
  pages={186--204},
  year={2008},
  publisher={Elsevier}
}

@article{fischetti2005feasibility,
  title={The feasibility pump},
  author={Fischetti, Matteo and Glover, Fred and Lodi, Andrea},
  journal={Mathematical programming},
  volume={104},
  number={1},
  pages={91--104},
  year={2005},
  publisher={Springer}
}

@article{shen2012nonmonotone,
  title={A nonmonotone filter method for nonlinear optimization},
  author={Shen, Chungen and Leyffer, Sven and Fletcher, Roger},
  journal={Computational Optimization and Applications},
  volume={52},
  number={3},
  pages={583--607},
  year={2012},
  publisher={Springer}
}
\endgroup
\end{document}